\documentclass[11pt,twoside]{article}

\usepackage{amsfonts,amssymb,amsmath,amsthm,mathtools}
\allowdisplaybreaks[4]
\usepackage{booktabs,array}
\usepackage{enumitem}
\usepackage{geometry}
\usepackage{changepage}
\usepackage{fancyhdr}
\usepackage{microtype}
\usepackage[hidelinks]{hyperref}
\hypersetup{
  pdfauthor={Jiangsheng Hu, Yu-Zhe Liu, Tiwei Zhao},
  pdftitle={Support tau-tilting posets and Hochschild reconstruction for
  centralizer matrix algebras}
}

\setenumerate[1]{itemsep=0pt,partopsep=0pt,parsep=\parskip,topsep=3pt}
\setitemize[1]{itemsep=0pt,partopsep=0pt,parsep=\parskip,topsep=3pt}
\setdescription{itemsep=0pt,partopsep=0pt,parsep=\parskip,topsep=3pt}
\setlist[itemize]{leftmargin=35pt}
\setlist[enumerate]{leftmargin=35pt}

\numberwithin{equation}{section}
\numberwithin{figure}{section}
\numberwithin{table}{section}

\def\EnglishTitle{Support $\tau$-tilting posets and Hochschild reconstruction for
centralizer matrix algebras}
\def\headertitle{Support $\tau$-tilting posets and Hochschild reconstruction}
\def\PaperAuthorsENname{Jiangsheng Hu$^{1}$,
Yu-Zhe Liu$^{2}$ and
Tiwei Zhao$^{3,\ddagger}$}

\theoremstyle{plain}
\newtheorem{theorem}{Theorem}[section]
\newtheorem{lemma}[theorem]{Lemma}
\newtheorem{proposition}[theorem]{Proposition}
\newtheorem{corollary}[theorem]{Corollary}
\theoremstyle{definition}
\newtheorem{definition}[theorem]{Definition}
\newtheorem{remark}[theorem]{Remark}

\newcommand{\stt}{\mathrm{s}\tau\text{-}\mathrm{tilt}}
\newcommand{\Fac}{\operatorname{Fac}}
\newcommand{\tors}{\operatorname{tors}}
\newcommand{\rad}{\operatorname{rad}}
\newcommand{\add}{\operatorname{add}}
\newcommand{\Irr}{\operatorname{Irr}}
\newcommand{\Weak}{\operatorname{Weak}}
\DeclareMathOperator{\End}{End}
\DeclareMathOperator{\Hom}{Hom}
\DeclareMathOperator{\fdeg}{fdeg}
\DeclareMathOperator{\proj}{proj}
\newcommand{\Kb}{K^{\mathrm{b}}}
\newcommand{\cO}{\mathcal{O}}
\newcommand{\ms}[1]{\{\!\{#1\}\!\}}
\newcommand{\simTZ}{\stackrel{\mathrm{TZ}}{\sim}}
\newcommand{\simbTZ}{\stackrel{\mathrm{bTZ}}{\sim}}
\newcommand{\simM}{\stackrel{\mathrm{M}}{\sim}}
\newcommand{\simD}{\stackrel{\mathrm{D}}{\sim}}
\newcommand{\simAD}{\stackrel{\mathrm{AD}}{\sim}}
\newcommand{\simT}{\stackrel{\mathrm{T}}{\sim}}
\newcommand{\simgTZ}{\stackrel{\mathrm{gTZ}}{\sim}}
\DeclareMathOperator{\coker}{coker}
\DeclareMathOperator{\domdim}{domdim}
\DeclareMathOperator{\gldim}{gl.dim}
\newcommand{\Dsg}{D_{\mathrm{sg}}}

\title{\bfseries \EnglishTitle}
\author{\PaperAuthorsENname}
\date{}

\begin{document}

\maketitle
\thispagestyle{empty}

\begin{center}
\footnotesize
$^{1}$School of Mathematics, Hangzhou Normal University,
Hangzhou 311121, P. R. China\\[-0.5mm]
E-mail: \href{mailto:hujs@hznu.edu.cn}{hujs@hznu.edu.cn}

\vspace{0.2mm}
$^{2}$School of Mathematics and Statistics, Guizhou University,
Guiyang 550025, P. R. China\\[-0.5mm]
E-mail: \href{mailto:liuyz@gzu.edu.cn}{liuyz@gzu.edu.cn}

\vspace{0.2mm}
$^{3}$School of Artificial Intelligence, Jianghan University,
Wuhan 430056, P. R. China\\[-0.5mm]
E-mail: \href{mailto:tiweizhao@jhun.edu.cn}{tiweizhao@jhun.edu.cn}

\vspace{0.2mm}
$^{\ddagger}$Corresponding author
\end{center}

\vspace{1mm}

\begin{adjustwidth}{1cm}{1cm}
\noindent\footnotesize
\textbf{Abstract}: Let $R$ be a field and $A$ the endomorphism algebra of a finite direct sum of cyclic modules over a finite-dimensional commutative local principal ideal $R$-algebra. An explicit central quotient identifies the support $\tau$-tilting poset of $A$ with the weak order on a symmetric group. The center of $A$, together with degree-zero Hochschild homology as a module over the center, recovers the truncated local algebra and the multiset of successive length gaps. Applied to centralizer matrix algebras, these results show that the support $\tau$-tilting poset determines exactly the multiset of distinct-exponent counts of the primary blocks, whereas the center--Hochschild pair recovers their local algebras and gap multisets. Combined with the known classification, this characterizes derived equivalence by isomorphism of these algebra--module pairs. We also establish Morita reconstruction results in the string and gentle classes.

\vspace{1mm}

\vspace{1mm}
\end{adjustwidth}

\section{Introduction}\label{sec:intro}

Reconstruction from tilting-theoretic posets asks which algebraic data an
order determines. Happel and Unger reconstructed certain acyclic path
algebras from their tilting posets \cite{HappelUnger}. Aihara and
Kase obtained support $\tau$-tilting reconstruction results for other
classes, including tree-quiver and weak-order cases
\cite{AiharaKase,KaseWeak,KaseInverse}. Kase further studied which
algebraic structures are forced to agree for algebras sharing the same
support $\tau$-tilting poset \cite{KaseInverseII}. Support $\tau$-tilting modules are
related to functorially finite torsion classes and basic two-term silting
complexes by the correspondences of Adachi, Iyama and Reiten
\cite[Theorems~2.7 and~3.2]{AIR}; hence their posets provide a natural
reconstruction invariant. In this paper, we will study endomorphism algebras of direct sums
of cyclic modules over local principal ideal algebras and, through primary
decomposition, centralizer matrix algebras. We determine the data recovered
by their support $\tau$-tilting posets and the finer invariant given by
the center together with degree-zero Hochschild homology.

Throughout this paper, $R$ denotes a field. By an algebra we mean a
finite-dimensional associative unital $R$-algebra, and modules are
finitely generated left modules. All algebra homomorphisms preserve
identities. Unless stated otherwise, isomorphisms and equivalences of
$R$-algebras are understood to be $R$-linear.

For an algebra $A$, let $\rad A$ denote its Jacobson radical. Its Loewy length is defined by
\[
 \operatorname{LL}(A):=
 \min\{\ell\geq1\mid(\rad A)^\ell=0\}.
\]
For $m\geq0$, write $\rad^m A=(\rad A)^m$, with $\rad^0 A=A$;
for an integer $r$, write $(r)_+:=\max\{r,0\}$.
For an object $X$ of an additive category, $\add X$ denotes the full
subcategory of direct summands of finite direct sums of copies of $X$.
We write $\ms{x_1,\ldots,x_r}$ for a multiset, with
repetitions retained.

For a nonempty finite set
$P=\{p_1<\cdots<p_s\}\subseteq\mathbb Z_{>0}$, define
$g_1:=p_1$, $g_i:=p_i-p_{i-1}$ for $i\ge2$, and
$H(P):=\ms{g_1,\ldots,g_s}$; we call $H(P)$ the \emph{gap multiset} of
$P$.
Let $\cO$ be a finite-dimensional commutative local principal ideal
$R$-algebra with $\rad\cO=(\pi)$ and $\operatorname{LL}(\cO)=\ell$. For
such a set $P\subseteq\{1,\ldots,\ell\}$, define
$\Lambda_{\cO}(P):=\End_{\cO}(\bigoplus_{i=1}^{s}\cO/(\pi^{p_i}))$ and
$\cO_P:=\cO/(\pi^{p_s})$. For an $R$-algebra $A$, write $\stt A$ for the
poset of isomorphism classes of basic support $\tau$-tilting $A$-modules,
ordered by $M\leq N$ if and only if $\Fac M\subseteq\Fac N$, where
$\Fac M$ is the full subcategory of factor modules of finite direct sums
of copies of $M$. Write $\tors A$ for the lattice of torsion classes in
$A\text{-}\mathrm{mod}$, $Z(A)$ for the center, and $[A,A]$ for the
$R$-subspace spanned by commutators. The degree-zero Hochschild homology $HH_0(A):=A/[A,A]$
is naturally a $Z(A)$-module. For $q\ge2$, let
$\Weak(\Sigma_q)$ denote the left weak order on the symmetric group
$\Sigma_q$.
{By an \emph{algebra--module pair} we mean a pair $(C,M)$
consisting of an $R$-algebra $C$ and a $C$-module $M$. An isomorphism
$(C,M)\simeq(C',M')$ means an $R$-algebra isomorphism
$\varphi\colon C\to C'$ together with an $R$-linear isomorphism
$\psi\colon M\to M'$ such that
$\psi(zm)=\varphi(z)\psi(m)$ for all $z\in C$ and $m\in M$.}

The central-reduction theorem of Eisele, Janssens and Raedschelders
\cite[Theorem~11]{EJR}, in the arbitrary-field form supplied by Kimura
\cite{Kimura}, is our starting point; Section~\ref{sec:prelim} gives a
direct proof. Iyama and Zhang obtained the weak-order formula
for the Auslander algebra corresponding to $P=\{1,\ldots,s\}$
\cite[Corollary~1.4]{IZ}. Over an algebraically closed field, Kase's
criterion also covers the split local cases for arbitrary $P$
\cite[Theorem~3.3]{KaseWeak}. Our local result combines an explicit
central quotient over an arbitrary field with a computation of
$HH_0$ as a module over the center.

\begin{theorem}\label{thm:intro-local-reconstruction}
Under the above assumptions, there are isomorphisms
\begin{equation}\label{eq:intro-local-reconstruction}
 \begin{aligned}
 \stt\Lambda_{\cO}(P)&\simeq\Weak(\Sigma_{s+1})
     &&\text{as posets},\\
 Z(\Lambda_{\cO}(P))&\simeq\cO_P
     &&\text{as $R$-algebras},\\
 HH_0(\Lambda_{\cO}(P))&\simeq
     \bigoplus_{i=1}^{s}\cO_P/(\pi^{g_i})
     &&\text{as $\cO_P$-modules}.
 \end{aligned}
\end{equation}
Consequently, the abstract support $\tau$-tilting poset determines exactly
$s=|P|$, whereas the algebra--module pair
$(Z(\Lambda_{\cO}(P)),HH_0(\Lambda_{\cO}(P)))$ determines $\cO_P$ and
the gap multiset $H(P)$.
\end{theorem}

For $c\in M_n(R)$, let
$S_n(c,R):=\{a\in M_n(R)\mid ac=ca\}$ be the centralizer matrix algebra of
$c$. The identification $S_n(c,R)\simeq\End_{R[x]}(R^n)$, where $x$ acts
as $c$, allows us to apply the local theorem blockwise. Although every
finite-dimensional algebra is the common centralizer of two matrices
(see \cite[Lemma~2]{Brenner}), centralizers of a single matrix retain enough
structure for an explicit reconstruction theory.

Let $\Irr(c)$ be the set of monic irreducible factors of the minimal
polynomial of $c$. Regard $V_c:=R^n$ as an $R[x]$-module by letting $x$
act as $c$, and let $V_c(f)$ denote its $f$-primary component. Then
\[
 S_n(c,R)\simeq
 \prod_{f\in\Irr(c)}\End_{R[x]}(V_c(f)).
\]
For $f\in\Irr(c)$, let
$P_c(f):=\{a\ge1\mid f^a\text{ is an elementary divisor of }c\}$, with
repetitions omitted, and set $r_c(f):=\max P_c(f)$,
$\kappa_c(f):=|P_c(f)|$, and $U_c(f):=R[x]/(f^{r_c(f)})$. Define the
\emph{$T$-type} of $c$ by
$T_R(c):=\ms{\kappa_c(f)\mid f\in\Irr(c)}$. We call the factor
$\End_{R[x]}(V_c(f))$ the \emph{$f$-primary block} of $S_n(c,R)$.

\begin{theorem}
\label{thm:main-recover}
There is an isomorphism of posets
\[
 \stt S_n(c,R)\simeq
 \prod_{f\in\Irr(c)}\Weak(\Sigma_{\kappa_c(f)+1}).
\]
In particular, $S_n(c,R)$ is $\tau$-tilting finite and
$|\stt S_n(c,R)|=\prod_{f\in\Irr(c)}(\kappa_c(f)+1)!$.
If $c\in M_n(R)$ and $d\in M_m(R)$, then the following are equivalent:
\begin{enumerate}[label=\textup{(\arabic*)}]
\item $T_R(c)=T_R(d)$;
\item $\stt S_n(c,R)\simeq\stt S_m(d,R)$ as posets;
\item {$\tors S_n(c,R)\simeq\tors S_m(d,R)$ as lattices.}
\end{enumerate}
\end{theorem}

The center supplies complementary block data. The classical
double-centralizer theorem gives $Z(S_n(c,R))=R[c]$ and
$R[c]\simeq\prod_{f\in\Irr(c)}U_c(f)$ as $R$-algebras
(see {\cite[Chapter~III, Theorem~3.17 and Corollary~1]{Jacobson}}); this
formula is also used in \cite{XiZhang1,XiZhang2}.
The local Hochschild formula extends to these primary blocks as follows.

\begin{theorem}
\label{thm:intro-hochschild}
There is an isomorphism of $Z(S_n(c,R))$-modules
\[
HH_0(S_n(c,R))\simeq
\bigoplus_{f\in\Irr(c)}\ \bigoplus_{g\in H(P_c(f))}
U_c(f)/\rad^gU_c(f).
\]
The inner direct sum is indexed with multiplicity by the gap multiset.
Here $Z(S_n(c,R))\simeq\prod_{f\in\Irr(c)}U_c(f)$ acts on each summand
through the projection to its $f$-factor. Consequently, the pair
$(Z(S_n(c,R)),HH_0(S_n(c,R)))$ recovers the local algebra and the gap
multiset of each primary block, up to isomorphism and reordering of the
blocks. Moreover,
$\dim_R HH_0(S_n(c,R))=
\sum_{f\in\Irr(c)}\deg(f)r_c(f)=\dim_R Z(S_n(c,R))$.
\end{theorem}

{The center-module decomposition refines the dimension identity
known over the field of complex numbers from the abelianization of Lie centralizers
\cite[Proposition~3.5]{Izosimov}, and holds over an arbitrary field.}

Together with derived invariance of the cap-product action
\cite[Theorem~2.2]{ArmentaKeller} and Li--Xi's classification theorem
\cite[Theorem~1.1]{LiXiD}, Theorem~\ref{thm:intro-hochschild} yields the
following characterization: for $d\in M_m(R)$, the centralizer matrix
algebras $S_n(c,R)$ and $S_m(d,R)$ are derived equivalent as $R$-algebras
if and only if the algebra--module pairs
$(Z(S_n(c,R)),HH_0(S_n(c,R)))$ and
$(Z(S_m(d,R)),HH_0(S_m(d,R)))$ are isomorphic; see
Corollary~\ref{cor:hochschild-derived-characterization}.

The reconstruction relations defined in Subsection~\ref{subsec:comparison}
are summarized below. The relations M, AD, and D denote Morita, almost
$\nu$-stable derived, and derived equivalence, respectively, where $\nu$
is the Nakayama functor \cite[Theorem~1.1]{LiXiD}. 
{
\begin{center}
\small
\renewcommand{\arraystretch}{1.08}
\begin{tabular}{@{}>{\bfseries}l p{0.76\textwidth}@{}}
\toprule
Relation & Invariant or equivalence data \tabularnewline
\midrule
M   & the local algebra $U_c(f)$ and the exponent set $P_c(f)$; \\
AD  & $U_c(f)$ and the exponent set up to the Li--Xi involution described in Subsection~\ref{subsec:comparison}; \\
D   & $U_c(f)$ and the gap multiset $H(P_c(f))$, equivalently the center--Hochschild pair; \\
gTZ & $U_c(f)$, $\kappa_c(f)$, and the Grothendieck group of the corresponding singularity category; \\
bTZ & $U_c(f)$ and $\kappa_c(f)$; \\
TZ  & the global center and the global support $\tau$-tilting poset, recorded separately; \\
T   & the multiset of the numbers $\kappa_c(f)$, equivalently the support $\tau$-tilting poset. \\
\bottomrule
\end{tabular}
\end{center}
}
Theorem~\ref{thm:strict-comparison} and Remark~\ref{rem:full-strict-chain}
establish the strict implications
\[
\mathrm{M}\Rightarrow\mathrm{AD}\Rightarrow\mathrm{D}
\Rightarrow\mathrm{gTZ}\Rightarrow\mathrm{bTZ}
\Rightarrow\mathrm{TZ}\Rightarrow\mathrm{T}
\]
over every field; the first two are illustrated by examples from
Li--Xi's classification.

We finally determine when these data recover the Morita class in the
string and gentle classes, using the conventions of
Section~\ref{sec:string} and Definition~\ref{def:string-gentle-scope}.

For a monic irreducible polynomial $f\in R[x]$ and a nonempty finite set
$P=\{p_1<\cdots<p_s\}$ of positive integers, {write
$K_f:=R[x]/(f)$ and define
$\Lambda_f(P):=\End_{R[x]}(\bigoplus_{p\in P}R[x]/(f^p))$.}
This is the basic Morita representative of a primary block with
irreducible factor $f$ and distinct exponent set $P$.

{Part~\textup{(1)} of the following theorem is a specialization of
Marczinzik's classification of monomial algebras of dominant dimension
at least two \cite{MarczinzikMonomial}; part~\textup{(2)} follows by imposing
the quadratic-relation condition. We include direct proofs and use these
criteria for the Morita reconstruction in part~\textup{(3)}.}

\begin{theorem}\label{thm:collapse}
Let $f\in R[x]$ be monic and irreducible, and let $P$ be a nonempty finite
set of positive integers.
\begin{enumerate}[label=\textup{(\arabic*)}]
\item $\Lambda_f(P)$ is Morita equivalent to a string
$R$-algebra if and only if $f$ is linear and
$P=\{p\}$ or $P=\{p,p+1\}$ for some $p\ge1$.
\item $\Lambda_f(P)$ is Morita equivalent to a gentle
$R$-algebra if and only if $f$ is linear and
$P\in\{\{1\},\{2\},\{1,2\}\}$.
\item For centralizer matrix algebras within the Morita-string class,
\[
 \mathrm{M}\Longleftrightarrow\mathrm{AD}\Longleftrightarrow
 \mathrm{D}\Longleftrightarrow\mathrm{gTZ}\Longleftrightarrow
 \mathrm{bTZ},
\]
where TZ is strictly weaker than bTZ and T is strictly weaker than TZ.
Within the Morita-gentle class, TZ determines the Morita class.
\end{enumerate}
\end{theorem}

The paper is organized as follows. In Section~\ref{sec:prelim}, we prove the central-reduction theorem. In
Section~\ref{sec:uniform}, we establish the local reconstruction formulas.
In Section~\ref{sec:centralizer}, we treat centralizer matrix algebras and
compares the reconstruction relations. In Section~\ref{sec:string}, we prove
the string and gentle criteria and their Morita-reconstruction consequences.

\section{Central reduction for support
\texorpdfstring{$\tau$}{tau}-tilting posets}\label{sec:prelim}

\subsection{Support \texorpdfstring{$\tau$}{tau}-tilting posets}
\label{subsec:tau-prelim}

Let $A$ be an algebra and $X$ an $A$-module. Denote by $|X|$ the number
of isomorphism classes of indecomposable direct summands of $X$, and by
$|A|$ the number of isomorphism classes of simple $A$-modules. Write
$A\text{-}\mathrm{mod}$ for the category of finitely generated left
$A$-modules and $\proj A$ for its full subcategory of projective modules.
Following \cite{AIR}, a pair
$
 (M,P)\in (A\text{-}\mathrm{mod})\times\proj A
$
is called \emph{$\tau$-rigid} if
$
 \Hom_A(M,\tau M)=0
 ~\text{and}~
 \Hom_A(P,M)=0.
$
{Here $\tau$ denotes the Auslander--Reiten translation.}
It is called a \emph{support $\tau$-tilting pair} if it is $\tau$-rigid and
$
 |M|+|P|=|A|.
$
In this case, $M$ is called a \emph{support $\tau$-tilting $A$-module}.
A torsion class in $A\text{-}\mathrm{mod}$ is a full subcategory closed
under extensions and factor modules; $\tors A$ is ordered by inclusion.

{Denote by $\Kb(\proj A)$ the bounded homotopy category of
$\proj A$, and by $K_0(\proj A)$ the split Grothendieck group of
$\proj A$, canonically identified with $K_0(\Kb(\proj A))$.}
A \emph{two-term complex} is a complex
\[
 C:=(P^{-1}\longrightarrow P^0)\in\Kb(\proj A)
\]
concentrated in degrees $-1$ and $0$. It is \emph{presilting} if
$
 \Hom_{\Kb(\proj A)}(C,C[i])=0
 ~\text{for every }i>0,
$
and it is \emph{silting} if, in addition,
$
 \operatorname{thick}(C)=\Kb(\proj A),
$
where $\operatorname{thick}(C)$ is the smallest thick subcategory
containing $C$. Its \emph{$g$-vector} is
\[
 {\mathbf g(C):=[P^0]-[P^{-1}]\in K_0(\proj A).}
\]

For basic two-term silting complexes $C$ and $D$, we write
$C\leq D$ if
$\Hom_{\Kb(\proj A)}(D,C[i])=0$ for all $i>0$.
Adachi, Iyama and Reiten established bijections over algebraically closed
fields between basic support
$\tau$-tilting modules, functorially finite torsion classes, and basic
two-term silting complexes. These bijections are given by
$M\mapsto\Fac M$ and $C\mapsto H^0(C)$, respectively
(see \cite[Theorems~2.7 and~3.2]{AIR}), and the latter is a poset isomorphism by
\cite[Corollary~3.9]{AIR}. These correspondences hold over arbitrary
fields by \cite{DIJ}.

\subsection{Central reduction}

{The following form of the central-reduction theorem of
Eisele, Janssens and Raedschelders \cite[Theorem~11]{EJR} is a special
case of Kimura's reduction theorem \cite{Kimura}. Indeed, if
$N=Z(A)\cap\rad A$, then $C=R\cdot1+N\subseteq Z(A)$ is a
commutative local Artin ring with radical $N$, hence a complete local
Noetherian ring. The algebra $A$ is finite over $C$, and the ideals
considered below lie in $NA$. We give a direct proof using the
Hom-space comparison of \cite{EJR}.}

\begin{proposition}
\label{prop:reduction}
Let \(A\) be a finite-dimensional algebra over an arbitrary field, and
set $J_A:=\bigl(Z(A)\cap\rad A\bigr)A$. If \(I\) is an ideal of
\(A\) such that \(I\subseteq J_A\), then reduction modulo \(I\)
induces an isomorphism of posets $\stt A\simeq\stt(A/I)$.
\end{proposition}

\begin{proof}
We first treat the case $I=(z)$, where
$z\in Z(A)\cap\rad A$ and $z^2=0$, and
{write $\overline A:=A/(z)$. For every $A$-module or complex
$X$, write $\overline X:=\overline A\otimes_A X$.}
For projective $A$-modules $X$ and $Y$, the canonical map
\[
 \Hom_A(X,Y)\longrightarrow
 {\Hom_{\overline A}(\overline X,\overline Y)}
\]
is surjective with kernel $z\Hom_A(X,Y)$ by projectivity of $X$ and
centrality of $z$. For two-term complexes $C$ and $D$, the calculation in
\cite[Equation~(4.2)]{EJR} gives
\begin{equation}\label{eq:central-hom-comparison}
 \Hom_{\Kb(\proj A)}(C,D[1])=0
 \quad\Longleftrightarrow\quad
 {\Hom_{\Kb(\proj\overline A)}
 (\overline C,\overline D[1])=0.}
\end{equation}
Since idempotents and morphisms between projective modules lift modulo
the square-zero ideal $(z)$, every two-term presilting complex over
{$\overline A$} lifts to one over $A$ by
\eqref{eq:central-hom-comparison}.

{For injectivity, let $C$ and $D$ be two-term presilting complexes
with $\overline C\simeq\overline D$. Then
$\overline C\oplus\overline D$ is presilting, so
\eqref{eq:central-hom-comparison} implies that $C\oplus D$ is
presilting. By Bongartz completion, there is a basic two-term silting
complex $T$ such that $C\oplus D\in\add T$. The classes of the
indecomposable summands of $T$ form a basis of $K_0(\proj A)$; both
facts hold over arbitrary fields by \cite{DIJ}. Reduction identifies
$K_0(\proj A)$ with $K_0(\proj\overline A)$ and preserves
$g$-vectors, so $\mathbf g(C)=\mathbf g(D)$. The multiplicities of
all indecomposable summands of $T$ in $C$ and $D$ therefore agree,
which gives $C\simeq D$. Hence reduction induces a bijection on
isomorphism classes of two-term presilting complexes.}
It also preserves and reflects direct-sum decompositions: if
{$\overline C\simeq\overline U\oplus\overline V$}, then lifts
$U$ and $V$ have a presilting direct sum by
\eqref{eq:central-hom-comparison}, and injectivity yields
$C\simeq U\oplus V$.
{By Bongartz completion and the basis theorem \cite{DIJ}, a basic
two-term presilting complex is silting if and only if it has $|A|$
indecomposable summands. Since $|A|=|\overline A|$, the bijection
restricts to basic two-term silting complexes.} Moreover,
\eqref{eq:central-hom-comparison} shows that it preserves and reflects
the silting order.

Now let $z\in Z(A)\cap\rad A$ be arbitrary. Choose an integer $n>2$ such that
$z^n=0$. For every $2\leq r\leq n$, the kernel of
$A/(z^r)\to A/(z^{r-1})$ is generated by the square-zero central
radical element $z^{r-1}+(z^r)$. Thus the square-zero case applies
successively and yields the assertion for reduction modulo $(z)$.

Choose $z_1,\ldots,z_m\in Z(A)\cap\rad A$ such that
$J_A=(z_1,\ldots,z_m)$. Successive principal reductions identify the
two-term silting posets of $A$ and $A/J_A$. Since
$I\subseteq J_A\subseteq\rad A$, the images of the $z_i$ in $A/I$ are
central radical elements generating $J_A/I$. The same argument identifies
the two-term silting posets of $A/I$ and $A/J_A$.
The factorization $A\to A/I\to A/J_A$ therefore shows that reduction
$A\to A/I$ induces a poset isomorphism. {The correspondence
between basic two-term silting complexes and support $\tau$-tilting
modules \cite{DIJ} gives $\stt A\simeq\stt(A/I)$.}
\end{proof}

\section{Reconstruction over local principal ideal algebras}
\label{sec:uniform}

We prove Theorem~\ref{thm:intro-local-reconstruction} by constructing
a central quotient and computing the center and degree-zero Hochschild
homology.

Let $\cO$ be a finite-dimensional commutative local principal ideal
$R$-algebra.  Write its maximal ideal as $(\pi)$, let
$\operatorname{LL}(\cO)=\ell$, and set
$k:=\cO/(\pi)$.  Thus
\[
 \cO=(\pi^0)\supset(\pi)\supset\cdots\supset(\pi^\ell)=0
 \ \ \text{and} \ \
 \pi^{\ell-1}\ne0.
\]
Let
$P=\{p_1<\cdots<p_s\}\subseteq\{1,\ldots,\ell\}$ be nonempty, and set
\[
 M(i):=\cO/(\pi^{p_i}),\qquad
 M_P:=\bigoplus_{i=1}^{s}M(i).
\]
Thus $\Lambda_{\cO}(P)=\End_{\cO}(M_P)$.
Multiplication by $\pi$ defines a central endomorphism
$z_\pi\in\Lambda_{\cO}(P)$ with $z_\pi^\ell=0$. Hence
$z_\pi\in Z(\Lambda_{\cO}(P))\cap\rad\Lambda_{\cO}(P)$.
We write $(z_\pi)$ for the two-sided ideal of
$\Lambda_{\cO}(P)$ generated by $z_\pi$.

{We compose paths and maps from right to left.}
Let $Q_s$ be the quiver with vertex set $\{1,\ldots,s\}$ and arrows
$\alpha_i:i\to i+1$ and $\beta_i:i+1\to i$ for $1\leq i<s$. Thus
$Q_s$ is the doubled line
\[
 {
 \underset{1}{\bullet}
 \underset{\beta_1}{\overset{\alpha_1}{\rightleftarrows}}
 \underset{2}{\bullet}\rightleftarrows\cdots\rightleftarrows
 \underset{s-1}{\bullet}
 \underset{\beta_{s-1}}{\overset{\alpha_{s-1}}{\rightleftarrows}}
 \underset{s}{\bullet}.}
\]
Set
$I_s:=\langle\beta_i\alpha_i,\alpha_i\beta_i\mid1\leq i<s\rangle$
and $B_s(k):=kQ_s/I_s$. Thus $B_1(k)=k$.

When $\cO=k[t]/(t^s)$ and $P=\{1,\ldots,s\}$,
$\Lambda_{\cO}(P)$ is the Auslander algebra of $k[t]/(t^s)$, whose
support $\tau$-tilting poset was determined by Iyama--Zhang
\cite[Corollary~1.4]{IZ}. Up to the path-composition convention,
$B_s(k)$ is the case $m=1$ of \cite[Example~3.2(3)]{KaseWeak}.
When $k$ is algebraically closed, \cite[Theorem~3.3]{KaseWeak} gives
the weak-order formula for $B_s(k)$ and also for
$\Lambda_{k[t]/(t^\ell)}(P)$, since the latter satisfy
Condition~3.1 of \cite{KaseWeak}. The following theorem constructs
an explicit central quotient for arbitrary $\cO$ and $P$.

\begin{theorem}
\label{thm:uniform-quotient}
The quotient $\Lambda_{\cO}(P)/(z_\pi)$ carries a canonical
$k$-algebra structure, and
\[
 \Lambda_{\cO}(P)/(z_\pi)\simeq B_s(k)
\]
as $k$-algebras.  In particular, the quotient depends on $\cO$ and
$P$ only through the residue field $k$ and the number
$s=|P|$ of distinct exponents.
\end{theorem}

\begin{proof}
Let $A:=\Lambda_{\cO}(P)$, $J:=(z_\pi)$ and $\overline A:=A/J$. Scalar
multiplication gives a central homomorphism $\cO\to Z(A)$. Since
multiplication by $\pi$ becomes zero in $\overline A$, this homomorphism
induces a unital map
\(
 k=\cO/(\pi)\longrightarrow Z(\overline A).
\)
Since $J\subseteq\rad A$, the algebra $\overline A$ is nonzero, so this
map is injective and defines its canonical $k$-algebra structure.

For $1\le i,j\le s$, define
$u_{ji}\colon M(i)\to M(j)$ by
$u_{ji}(1):=\pi^{(p_j-p_i)_+}$. Since every
ideal of $\cO$ is a power of $(\pi)$, evaluation at $1$ gives
\[
\Hom_{\cO}(M(i),M(j))=\cO u_{ji}\simeq
\cO/(\pi^{\min\{p_i,p_j\}}).
\]
Let $e_i$ be the projection onto $M(i)$ followed by its inclusion into
$M_P$. Since $z_\pi$ is central,
$e_jJe_i=\pi\Hom_{\cO}(M(i),M(j))$, so the $(j,i)$-corner of
$\overline A$ has $k$-basis $\{\overline u_{ji}\}$. Thus
$\overline A=\bigoplus_{1\le i,j\le s}k\overline u_{ji}$ and
$\dim_k\overline A=s^2$.

{For $i,j,h\in\{1,\ldots,s\}$, direct evaluation at $1$ gives
$u_{hj}u_{ji}=\pi^{\delta(i,j,h)}u_{hi}$, where
$\delta(i,j,h):=(p_j-p_i)_+ +(p_h-p_j)_+ -(p_h-p_i)_+$. Since
$p_1<\cdots<p_s$, we have $\delta(i,j,h)=0$ if and only if
$i\le j\le h$ or $i\ge j\ge h$.
Otherwise, $\delta(i,j,h)\ge1$.} Thus, modulo $(z_\pi)$, monotone
compositions survive and every composition containing a change of direction
vanishes.

{Let $\varepsilon_i$ denote the trivial path at vertex $i$.}
The assignments
{$\varepsilon_i\mapsto\overline u_{ii}$},
$\alpha_i\mapsto\overline u_{i+1,i}$, and
$\beta_i\mapsto\overline u_{i,i+1}$
extend to a homomorphism $kQ_s\to\overline A$ that annihilates $I_s$
by the composition formula. The induced map $B_s(k)\to\overline A$ is
surjective because each $\overline u_{ji}$ is the image of the unique
monotone path from $i$ to $j$. Their residue classes form a $k$-basis of $B_s(k)$,
so both algebras have dimension $s^2$ and the map is an isomorphism.
\end{proof}

\begin{corollary}
\label{cor:weak-order-principle}
Let $\cO$ be a finite-dimensional commutative local principal ideal
$R$-algebra of Loewy length $\ell$, and let
$P=\{p_1<\cdots<p_s\}\subseteq\{1,\ldots,\ell\}$ be nonempty. Then
\[
\stt\Lambda_{\cO}(P)\simeq\Weak(\Sigma_{s+1})
\]
as posets. Consequently, the poset depends, up to isomorphism, only on
$s$, which is recovered as its number of atoms.
\end{corollary}

\begin{proof}
By Proposition~\ref{prop:reduction} and Theorem~\ref{thm:uniform-quotient}, 
we have 
\(
 \stt\Lambda_{\cO}(P)\simeq\stt B_s(k).
\)
Since $k$ is finite-dimensional over $R$, a $B_s(k)$-module is
finite-dimensional over $R$ if and only if it is finite-dimensional
over $k$. Hence computing
$\stt B_s(k)$ over $R$ or over $k$ gives the same poset.

Let
$\Gamma_s(k):=\End_{k[t]/(t^s)}
\bigl(\bigoplus_{i=1}^s k[t]/(t^i)\bigr)$,
the Auslander algebra of $k[t]/(t^s)$.
Applying Proposition~\ref{prop:reduction} and
Theorem~\ref{thm:uniform-quotient} over $k$ to the exponent set
$\{1,\ldots,s\}$ yields
$\stt B_s(k)\simeq\stt\Gamma_s(k)$.
The result follows from
$\stt\Gamma_s(k)\simeq\Weak(\Sigma_{s+1})$ by
\cite[Corollary~1.4(3),(4)]{IZ}.
{Those two parts identify the opposite left weak order with the
opposite generation order on right $\Gamma_s(k)^{\mathrm{op}}$-modules.
Reversing both orders gives precisely our convention for left
$\Gamma_s(k)$-modules.}
\end{proof}

We retain the notation $g_i$, $H(P)$ and $\cO_P$ from the Introduction.
Since $M(s)=\cO_P$,
scalar multiplication gives a faithful action of $\cO_P$ on $M_P$ and
central actions on $\Lambda_{\cO}(P)$ and its cocenter
$HH_0(\Lambda_{\cO}(P))$. We first compute the diagonal commutators.

\begin{lemma}
\label{lem:principal-diagonal-commutators}
Set $A:=\Lambda_{\cO}(P)$, and let $e_i$ be the idempotent
corresponding to the summand $M(i)$. Via scalar multiplication,
identify $e_iAe_i=\End_{\cO}(M(i))$ with
$\cO/(\pi^{p_i})e_i$, and set
$D:=\bigoplus_{i=1}^s e_iAe_i$ and
$N:=D\cap[A,A]$. Then the inclusion $D\hookrightarrow A$ induces an
isomorphism of $\cO_P$-modules
$D/N\xrightarrow{\sim}HH_0(A)$.
Moreover, writing $d_{ij}:=p_j-p_i$, the $\cO_P$-module $N$ is
generated by
\begin{equation}
\label{eq:principal-diagonal-commutators}
\begin{aligned}
&\pi^w(e_j-e_i)
&&\text{for }d_{ij}\leq w<p_i,\\
&\pi^we_j
&&\text{for }\max\{d_{ij},p_i\}\leq w<p_j,
\end{aligned}
\end{equation}
for all $i<j$, with empty ranges omitted.
\end{lemma}

\begin{proof}
The orthogonal idempotents give
$A=D\oplus\bigoplus_{i\neq j}e_jAe_i$. For $y\in e_jAe_i$ with
$i\neq j$, we have $y=[e_j,y]\in[A,A]$. Thus
$D\to A/[A,A]$ is surjective with kernel $N$, proving the first
assertion. Since the diagonal corners are commutative, the diagonal
component of any commutator is a sum of paired-corner commutators
$xy-yx$, with $x\in e_jAe_i$ and $y\in e_iAe_j$. Each such commutator
lies in $N$, so they generate $N$.

Fix $i<j$ and write $d:=d_{ij}$. The cyclic $\cO_P$-modules
$e_jAe_i$ and $e_iAe_j$ have generators $u_{ji}$ and $u_{ij}$,
where $u_{ji}(1)=\pi^d$ and $u_{ij}$ is the canonical projection.
For $a,b\in\cO_P$,
{$(au_{ji})(bu_{ij})=ab\pi^de_j$ and
$(bu_{ij})(au_{ji})=ab\pi^de_i$}, in the corresponding truncated diagonal corners. Taking $b=1$ shows
that these commutators generate $\cO_P\pi^d(e_j-e_i)$.
Its $\pi$-multiples are $\pi^w(e_j-e_i)$ for $d\leq w<p_j$:
the $e_i$-component vanishes when $w\geq p_i$, and both components
vanish when $w\geq p_j$. This gives the two ranges in
\eqref{eq:principal-diagonal-commutators}; summing over $i<j$ proves
the assertion.
\end{proof}

\begin{theorem}
\label{thm:principal-hh-reconstruction}
Let $\cO$ be a finite-dimensional commutative local principal ideal
$R$-algebra of Loewy length $\ell$, and let
$P=\{p_1<\cdots<p_s\}\subseteq\{1,\ldots,\ell\}$ be nonempty.
Then scalar multiplication induces an $R$-algebra isomorphism
$\cO_P\xrightarrow{\sim}Z(\Lambda_{\cO}(P))$, and there is an
isomorphism of $\cO_P$-modules
\begin{equation}\label{eq:principal-hh-reconstruction}
 HH_0\bigl(\Lambda_{\cO}(P)\bigr)
 \simeq \bigoplus_{i=1}^{s}\cO_P/(\pi^{g_i}).
\end{equation}
Consequently, the center-module structure of
$HH_0(\Lambda_{\cO}(P))$ determines $\ms{g_1,\ldots,g_s}$.
\end{theorem}

\begin{proof}
Write $A:=\Lambda_{\cO}(P)$, and let $e_i$ and $D$ be as in
Lemma~\ref{lem:principal-diagonal-commutators}. The faithful action
of $\cO_P$ on $M_P$ embeds $\cO_P$ into $Z(A)$.
Conversely, commutation with the $e_i$ forces every $a\in Z(A)$
to be diagonal, acting on $M(i)$ by multiplication by some
$q_i\in\cO/(\pi^{p_i})$. For $i<j$, commutation with the
canonical projection $M(j)\twoheadrightarrow M(i)$ gives
$q_j\equiv q_i\pmod{\pi^{p_i}}$. Taking $j=s$, it follows that every
$q_i$ is induced by $q_s\in\cO_P$, proving the asserted
$R$-algebra isomorphism $\cO_P\simeq Z(A)$.

It remains to compute $HH_0(A)$. By
Lemma~\ref{lem:principal-diagonal-commutators}, $HH_0(A)$ is the
quotient of {$D$} by the relations in
\eqref{eq:principal-diagonal-commutators}. Let $\overline e_i$ denote
the image of $e_i$ in $HH_0(A)$, and define
$h_1:=\overline e_1$ and
$h_i:=\overline e_i-\overline e_{i-1}$ for $2\leq i\leq s$. Since
$g_1=p_1$, the relation $\pi^{p_1}\overline e_1=0$ gives
$\pi^{g_1}h_1=0$. Let $i\geq2$. If $g_i<p_{i-1}$, then the first
relation in \eqref{eq:principal-diagonal-commutators} for the pair
$(i-1,i)$, with $w=g_i$, gives $\pi^{g_i}h_i=0$. If
$g_i\geq p_{i-1}$, then the second relation gives
$\pi^{g_i}\overline e_i=0$, while
$\pi^{g_i}\overline e_{i-1}=0$ follows from
$\pi^{p_{i-1}}\overline e_{i-1}=0$. Hence
$\pi^{g_i}h_i=0$ for every $i$.

Consider the module {$V:=\bigoplus_{i=1}^s\cO_P/(\pi^{g_i})\xi_i$}. Since
$\overline e_i=h_1+\cdots+h_i$, the assignments
{$\xi_i\mapsto h_i$}
define a surjective $\cO_P$-linear map
$\varphi\colon V\twoheadrightarrow HH_0(A)$. To construct its
inverse, consider the assignment
{$e_i\mapsto\xi_1+\cdots+\xi_i$}. Since
$p_i=g_1+\cdots+g_i$, the element $\pi^{p_i}$ annihilates
{$\xi_1,\ldots,\xi_i$}. Thus this assignment induces an
$\cO_P$-linear map {$\psi\colon D\to V$}. Let $i<j$. Since
$d_{ij}=g_{i+1}+\cdots+g_j$, if $w\geq d_{ij}$, then
$\psi(\pi^w(e_j-e_i))
={\pi^w(\xi_{i+1}+\cdots+\xi_j)}=0$. Hence $\psi$ annihilates the first
family in \eqref{eq:principal-diagonal-commutators}. If
$w\geq\max\{d_{ij},p_i\}$, then $w\geq g_r$ for every
$1\leq r\leq j$, and therefore
$\psi(\pi^we_j)={\pi^w(\xi_1+\cdots+\xi_j)}=0$. Thus $\psi$ also
annihilates the second family and factors through an $\cO_P$-linear
map $\overline\psi\colon HH_0(A)\to V$. Since
{$\overline\psi(h_i)=\xi_i$ and
$\varphi(\xi_1+\cdots+\xi_i)=\overline e_i$}, the maps $\varphi$ and
$\overline\psi$ are inverse. Hence
\eqref{eq:principal-hh-reconstruction} follows.

Each summand $\cO_P/(\pi^{g_i})$ has a local endomorphism ring and
Loewy length $g_i$. The Krull--Schmidt theorem therefore recovers
$\ms{g_1,\ldots,g_s}$ from $HH_0(A)$ as a module over
$Z(A)\simeq\cO_P$.
\end{proof}

Derived invariance of {the pair $(Z(A),HH_0(A))$} gives the
following consequence of Theorem~\ref{thm:principal-hh-reconstruction}.

\begin{corollary}
\label{cor:local-derived-reconstruction}
Let $\cO$ and $\cO'$ be finite-dimensional commutative local principal
ideal $R$-algebras with $\rad\cO=(\pi)$ and
$\rad\cO'=(\pi')$. Let
$P\subseteq\{1,\ldots,\operatorname{LL}(\cO)\}$ and
$Q\subseteq\{1,\ldots,\operatorname{LL}(\cO')\}$ be nonempty, and set
$A:=\Lambda_{\cO}(P)$ and $B:=\Lambda_{\cO'}(Q)$. If $A$ and $B$ are
$R$-linearly derived equivalent, then
$\cO/(\pi^{\max P})\simeq\cO'/(\pi'^{\max Q})$ as $R$-algebras and
$H(P)=H(Q)$. Consequently, $|P|=|Q|$ and
$\stt A\simeq\stt B$. In particular, the support $\tau$-tilting poset
is a derived invariant within this local family.
\end{corollary}

\begin{proof}
{Derived invariance of the cap product
\cite[Theorem~2.2]{ArmentaKeller}, applied with $M=A$ in degree zero,
yields an $R$-algebra isomorphism
$\varphi\colon Z(A)\xrightarrow{\sim}Z(B)$ and an $R$-linear
isomorphism $\psi\colon HH_0(A)\xrightarrow{\sim}HH_0(B)$ satisfying
$\psi(zm)=\varphi(z)\psi(m)$ for all $z\in Z(A)$ and
$m\in HH_0(A)$.}

By Theorem~\ref{thm:principal-hh-reconstruction}, $\varphi$ identifies
the centers $\cO/(\pi^{\max P})$ and $\cO'/(\pi'^{\max Q})$.
The same theorem decomposes $HH_0(A)$ and $HH_0(B)$ into indecomposable
modules with Loewy-length multisets $H(P)$ and $H(Q)$. Compatibility of
$\psi$ with $\varphi$ identifies these decompositions after transport of
scalars. The Krull--Schmidt theorem gives $H(P)=H(Q)$, hence $|P|=|Q|$
and, by Corollary~\ref{cor:weak-order-principle}, $\stt A\simeq\stt B$.
\end{proof}

We apply these results to the primary endomorphism algebras of
centralizer matrix algebras. Let $f\in R[x]$ be monic and irreducible,
and let $P=\{p_1<\cdots<p_s\}$ be nonempty. Define
$\cO_f:=R[x]/(f^{p_s})$ and $\pi_f:=f+(f^{p_s})$.
Then $\cO_f$ is a local principal ideal $R$-algebra with maximal ideal
$(\pi_f)$, Loewy length $p_s$, and residue field $K_f$.
{In $\Lambda_f(P)=\End_{R[x]}(\bigoplus_{i=1}^sR[x]/(f^{p_i}))$,}
let $z_f$ act by multiplication by $f$ on each summand.

\begin{corollary}
\label{cor:local-weak-order}
With the above notation,
$\Lambda_f(P)/(z_f)\simeq B_s(K_f)$ as $K_f$-algebras, and
\[
\stt\Lambda_f(P)\simeq\Weak(\Sigma_{s+1})
\]
as posets.
\end{corollary}

\begin{proof}
Since $f^{p_s}$ annihilates every summand $R[x]/(f^{p_i})$, the
$R[x]$-action factors through $\cO_f$. Hence there is a natural
$R$-algebra isomorphism
$\Lambda_f(P)\simeq\Lambda_{\cO_f}(P)$ carrying $z_f$ to
$z_{\pi_f}$. Therefore, Theorem~\ref{thm:uniform-quotient} gives the
first assertion, and Corollary~\ref{cor:weak-order-principle} gives
the second.
\end{proof}

\begin{remark}
\label{rem:residue-field-blindness}
{
The weak-order poset does not determine the residue field up to
$R$-algebra isomorphism.
Let $f,f'\in R[x]$ be monic and irreducible, and suppose that
$K_f\not\simeq K_{f'}$ as $R$-algebras. For
$K\in\{K_f,K_{f'}\}$,
$B_s(K)/\rad B_s(K)\simeq K^s$, so the endomorphism ring of every
simple $B_s(K)$-module is isomorphic to $K$. Since an $R$-linear
Morita equivalence preserves the endomorphism rings of simple modules,
$B_s(K_f)$ and $B_s(K_{f'})$ are not Morita equivalent as
$R$-algebras. Nevertheless, the proof of
Corollary~\ref{cor:weak-order-principle} gives
$\stt B_s(K_f)\simeq\Weak(\Sigma_{s+1})
\simeq\stt B_s(K_{f'})$.
}
\end{remark}

\section{Centralizer matrix algebras}\label{sec:centralizer}

In this section, we first extend the results of
Section~\ref{sec:uniform} to centralizer matrix algebras. We then compare
the resulting reconstruction relations and their homological consequences.

\subsection{Primary decomposition and recovery of the T-type}
\label{subsec:product-recovery}

{Let \(c\in M_n(R)\), and retain the \(R[x]\)-module
\(V_c\) introduced above. By the
elementary-divisor decomposition,
\[
\begin{aligned}
 V_c&\simeq\bigoplus_{f\in\Irr(c)}X_f,\qquad
 X_f:=V_c(f),\\
 X_f&\simeq\bigoplus_{a\in P_c(f)}
       (R[x]/(f^a))^{m_c(f,a)}.
\end{aligned}
\]
}Here $m_c(f,a)\geq1$ is the multiplicity
of $f^a$ among the elementary divisors of $c$.

We have $S_n(c,R)\simeq\End_{R[x]}(V_c)$. For
{distinct \(f,h\in\Irr(c)\) and $a,b\geq1$,}
{\[
 \Hom_{R[x]}(R[x]/(f^a),R[x]/(h^b))=0.
\]}
{Indeed, since \((f^a,h^b)=1\), multiplication by \(f^a\) is invertible on}
{\(R[x]/(h^b)\). If \(\varphi:R[x]/(f^a)\to R[x]/(h^b)\), then}
\(f^a\varphi(1)=0\), and hence \(\varphi=0\). Consequently,
\begin{equation}\label{eq:primary-product}
 S_n(c,R)
 \simeq
 \prod_{f\in\Irr(c)}\End_{R[x]}(X_f).
\end{equation}

For each \(f\in\Irr(c)\), let
$Y_f:=\bigoplus_{a\in P_c(f)}R[x]/(f^a)$. The module \(Y_f\) contains
one representative of each isomorphism class of indecomposable summands
of \(X_f\), so
$\add X_f=\add Y_f$.
The additive-generator form of Morita theory
\cite[Lemma~2.2]{LiXiD} gives
\begin{equation}\label{eq:additive-generator-reduction}
 \End_{R[x]}(X_f)
 {\ \text{is Morita equivalent to}\ }
 \End_{R[x]}(Y_f)
 =
 \Lambda_f(P_c(f)).
\end{equation}

\begin{proof}[\bf Proof of the product formula in
Theorem~\ref{thm:main-recover}]
Support \(\tau\)-tilting posets are Morita invariant. Moreover,
\((A_1\times A_2)\text{-}\mathrm{mod}\simeq
A_1\text{-}\mathrm{mod}\times A_2\text{-}\mathrm{mod}\), and the support
\(\tau\)-tilting conditions and the generation order are computed
componentwise. {Thus \eqref{eq:primary-product},
\eqref{eq:additive-generator-reduction}, and
Corollary~\ref{cor:local-weak-order} give
\[
 \stt S_n(c,R)\simeq
 \prod_{f\in\Irr(c)}\stt\Lambda_f(P_c(f))
 \simeq
 \prod_{f\in\Irr(c)}\Weak(\Sigma_{\kappa_c(f)+1}).
\]
Each factor is finite, so \(S_n(c,R)\) is \(\tau\)-tilting finite, and
\[
 |\stt S_n(c,R)|
 =\prod_{f\in\Irr(c)}(\kappa_c(f)+1)!.
\]}
\end{proof}

For \(s\geq1\), write $L_s:=\Weak(\Sigma_{s+1})$.
We recover the factor sizes from products of these lattices as follows.

Let \(L\) be a finite lattice with minimum element \(\hat 0\).  For \(x\leq y\) in \(L\), write
$[x,y]:=\{z\in L\mid x\leq z\leq y\}$ for the corresponding order
interval.  Recall that an atom of \(L\) is
an element covering \(\hat 0\).  We define the \emph{atom-interaction
graph} {\(\mathcal A(L)\)} to have the atoms of \(L\) as its vertices, with two
distinct atoms \(a,b\) joined by an edge if
\(
 \bigl|[\hat 0,a\vee b]\bigr|=6.
\)
An order isomorphism \(L\simeq L'\) preserves atoms, joins, and interval
cardinalities, and hence induces a graph isomorphism
{$\mathcal A(L)\simeq\mathcal A(L')$}.

\begin{lemma}
\label{lem:atom-graph}
The graph {\(\mathcal A(L_s)\)} is a path with \(s\) vertices.  More generally,
if $L=\prod_{r=1}^{q}L_{s_r}$, then the multiset of connected-component
sizes of {\(\mathcal A(L)\)} is $\ms{s_1,\ldots,s_q}$.
\end{lemma}

\begin{proof}
The atoms of \(L_s\) are the simple transpositions
\(
 \sigma_1,\ldots,\sigma_s.
\)
For distinct atoms \(a,b\), the interval \([\hat 0,a\vee b]\) is the weak
order on the rank-two parabolic subgroup \(\langle a,b\rangle\), which has
{twice the order of $ab$ elements}; see
\cite[Sections~3.1--3.2]{BjornerBrenti}. In type
\(A_s\), this number is six exactly for adjacent simple transpositions.
Therefore, {$\mathcal A(L_s)=
\sigma_1-\sigma_2-\cdots-\sigma_s$}.

In $L=\prod_{r=1}^{q}L_{s_r}$, every atom is supported in exactly one
factor. Atoms from distinct factors have a four-element interval below
their join, so they are not adjacent in {\(\mathcal A(L)\)}.
Within the \(r\)-th factor, the preceding argument gives a path with
\(s_r\) vertices. Thus the connected-component sizes are
$s_1,\ldots,s_q$, up to permutation.
\end{proof}

\begin{proof}[\bf Proof of the equivalences in
Theorem~\ref{thm:main-recover}]
The product formula gives \textup{(1)}$\Rightarrow$\textup{(2)}.

Conversely, suppose that $\stt S_n(c,R)\simeq\stt S_m(d,R).$
Their atom-interaction graphs are isomorphic. By the product formula and
Lemma~\ref{lem:atom-graph}, the multisets of factor sizes coincide, so
\(T_R(c)=T_R(d)\). This proves \textup{(2)}$\Rightarrow$\textup{(1)}.

It remains to compare \textup{(2)} and \textup{(3)}. By the finiteness
assertion proved above, both centralizer matrix algebras are \(\tau\)-tilting
finite. Therefore, every torsion class is
functorially finite by \cite[Theorem~3.8]{DIJ}.  The
Adachi--Iyama--Reiten correspondence consequently gives the poset
isomorphisms $\stt S_n(c,R)\simeq\tors S_n(c,R)$ and
$\stt S_m(d,R)\simeq\tors S_m(d,R)$, where a support \(\tau\)-tilting module \(M\) corresponds to
\(\Fac M\).

This proves \textup{(2)}$\Leftrightarrow$\textup{(3)}.
\end{proof}

\begin{remark}
\label{ex:scalar-extension}
Let $R=\mathbb Q$ and let $c$ be the companion matrix of $x^2+1$. Then
{$S_2(c,\mathbb Q)\simeq\mathbb Q(i)$}, so $\stt S_2(c,\mathbb Q)$ is the
two-element chain $\Weak(\Sigma_2)$. After scalar extension to $\mathbb C$,
the polynomial splits as $(x-i)(x+i)$. Hence the centralizer matrix algebra becomes
$\mathbb C\times\mathbb C$, and its support $\tau$-tilting poset is the
four-element Boolean lattice $\Weak(\Sigma_2)^2$. Thus
Theorem~\ref{thm:main-recover} cannot be obtained by passing formally to a
splitting field and descending.
\end{remark}

\subsection{Centers and their local factors}
\label{subsec:center}

The double-centralizer theorem and the Chinese remainder theorem give
the following description of the center and its local factors.
For related structural results on centralizer matrix algebras, see
\cite{XiZhang1,XiZhang2}.

\begin{proposition}
\label{prop:center}
Let \(c\in M_n(R)\), and let \(\mu_c\) be its minimal polynomial. Then
\[
        Z\bigl(S_n(c,R)\bigr)
        =
        R[c]
        {\simeq}
        R[x]/(\mu_c)
        {\simeq}
        \prod_{f\in\Irr(c)}
        U_c(f).
\]
\end{proposition}

\begin{proof}
{For an $R$-algebra $B$ and a subset $X\subseteq B$, write
$C_B(X):=\{b\in B\mid bx=xb\text{ for every }x\in X\}$; for
$x\in B$, write $C_B(x):=C_B(\{x\})$.}
Let
$\mathcal C:=S_n(c,R)=C_{M_n(R)}(c).$
Since \(c\in\mathcal C\), we have
\(C_{M_n(R)}(\mathcal C)\subseteq C_{M_n(R)}(c)=\mathcal C\). Thus
$
        Z(\mathcal C)
        =
        \mathcal C\cap C_{M_n(R)}(\mathcal C)
        =C_{M_n(R)}(\mathcal C)
        =
        C_{M_n(R)}\bigl(C_{M_n(R)}(c)\bigr).
$
By the classical double-centralizer theorem for a single matrix
{\cite[Chapter~III, Theorem~3.17 and Corollary~1]{Jacobson}},
the last algebra is \(R[c]\).

The evaluation homomorphism $R[x]\to R[c]$, $p(x)\mapsto p(c)$,
has kernel $(\mu_c)$, so $R[c]\simeq R[x]/(\mu_c)$.
Since $\mu_c=\prod_{f\in\Irr(c)}f^{r_c(f)}$ has pairwise coprime factors,
the Chinese remainder theorem gives {$R[x]/(\mu_c)\simeq
\prod_{f\in\Irr(c)}R[x]/(f^{r_c(f)})$}.
\end{proof}

\begin{corollary}
\label{cor:center-recovers}
The abstract algebra \(Z(S_n(c,R))\) determines the multiset
{$\ms{U_c(f)\mid f\in\Irr(c)}$} up to permutation and
\(R\)-algebra isomorphism.

For each such \(f\), the local algebra \(U_c(f)\) determines
$r_c(f)=\operatorname{LL}(U_c(f))$, its residue field
{$U_c(f)/\rad U_c(f)\simeq K_f$}, and
$\dim_R U_c(f)=r_c(f)\deg f$.
Consequently,
\[
        \dim_R Z\bigl(S_n(c,R)\bigr)
        =
        \deg\mu_c
        =
        \sum_{f\in\Irr(c)}r_c(f)\deg f.
\]
\end{corollary}

\begin{proof}
For each \(f\), the algebra \(U_c(f)\) is local, with
$\rad U_c(f)=(f)/(f^{r_c(f)})$.
Thus its Loewy length is \(r_c(f)\), and its residue field is
\(R[x]/(f)\).

Uniqueness of the decomposition into local factors and
Proposition~\ref{prop:center} show that the center determines the multiset
of the algebras \(U_c(f)\). Finally,
$\dim_R U_c(f)=\dim_R R[x]/(f^{r_c(f)})=r_c(f)\deg f$.
\end{proof}

\begin{remark}
\label{rem:U-iso}
{For monic irreducible polynomials $f\in R[x]$ and $g\in R[y]$
and positive integers $r,s$, one has
\[
 R[x]/(f^r)\simeq_R R[y]/(g^s)
 \quad\Longleftrightarrow\quad
 r=s\quad\text{and}\quad R[x]/(f)\simeq_R R[y]/(g).
\]
Here the residue fields are compared as $R$-algebras. Necessity follows
by passing to residue fields and comparing Loewy lengths. For sufficiency,
assume $r=s$ and fix an $R$-algebra isomorphism between the residue fields.
If this field extension is separable, both algebras are isomorphic
to the corresponding truncated polynomial algebra over that field by
\cite[Lemma~2.14]{LiXiD}.}

{In the inseparable case, $f'=0$. Choose polynomials $h(x)$ and $v(y)$
representing mutually inverse residue-field isomorphisms, with
$\bar y\mapsto\overline{h(x)}$ and
$\bar x\mapsto\overline{v(y)}$. Then, for suitable polynomials $a(x)$
and $b(y)$,
\[
 v(h(x))=x+f(x)a(x),\qquad f(v(y))=g(y)b(y).
\]
Since $f'=0$, these identities give
\[
 g(h(x))b(h(x))=f(x+f(x)a(x))\equiv f(x)\pmod{f(x)^2}.
\]
Also $f\mid g(h(x))$, so the $f$-adic order of $g(h(x))$ is exactly one.
Thus substitution $y\mapsto h(x)$ induces an $R$-algebra homomorphism
$R[y]/(g^r)\to R[x]/(f^r)$. If $q(y)=g(y)^j d(y)$ with $g\nmid d$,
the residue-field isomorphism implies $f\nmid d(h(x))$. Hence
$q(h(x))$ has $f$-adic order $j$, which proves that this homomorphism
is injective. The two algebras have the same $R$-dimension, so it is an
isomorphism. This argument does not require the residue-field extension
to be purely inseparable.}

{This isomorphism criterion does not assert the existence of an
$R$-algebra isomorphism $R[x]/(f^r)\simeq K_f[t]/(t^r)$ in the
inseparable case. For example, let $R=\mathbb F_p(u)$ and $f=x^p-u$.
The epimorphism $R[x]/(f^2)\twoheadrightarrow K_f$ has no $R$-algebra
section: every lift $x+fh$ of $\bar x$ satisfies
$(x+fh)^p-u\equiv f\ne0\pmod{f^2}$. Consequently,
$R[x]/(f^2)$ and $K_f[t]/(t^2)$ are not isomorphic as $R$-algebras,
although both are local principal ideal algebras with residue field
$K_f$ and Loewy length two. Thus residue field and Loewy length need
not determine the general local algebras of Section~\ref{sec:uniform},
even though they determine the polynomial quotients considered here.}
\end{remark}

\subsection{Hochschild reconstruction and derived invariance}
\label{subsec:derived}

The local Hochschild reconstruction yields the following global formula.

\begin{theorem}
\label{thm:hh-center-module}
Let $c\in M_n(R)$. Under the center isomorphism of
Proposition~\ref{prop:center}, there is an isomorphism of
$Z(S_n(c,R))$-modules
\[
HH_0(S_n(c,R))\simeq
\bigoplus_{f\in\Irr(c)}\ \bigoplus_{g\in H(P_c(f))}
U_c(f)/\rad^gU_c(f).
\]
The gaps are counted with multiplicity, and the center acts on each
summand through its $f$-factor.
\end{theorem}

\begin{proof}
{Under the Morita equivalence in
\eqref{eq:additive-generator-reduction}, the center isomorphism identifies
the scalar actions of $U_c(f)$ on $X_f$ and $Y_f$, and the induced
isomorphism on $HH_0$ respects these actions. Thus
Theorem~\ref{thm:principal-hh-reconstruction} gives the formula for each
primary block.} Compatibility of Hochschild homology and its center
action with finite products gives the global formula.
\end{proof}

Let $c\in M_n(R)$ and $d\in M_m(R)$. Following \cite[Definition~3.1(2)]{LiXiD}, we write
$c\simD d$ if there is a bijection $\sigma\colon\Irr(c)\to\Irr(d)$ such
that $U_c(f)\simeq U_d(\sigma(f))$ as $R$-algebras and
$H(P_c(f))=H(P_d(\sigma(f)))$ as multisets for every $f\in\Irr(c)$.

Together with Li--Xi's classification theorem, the center--module formula
gives the following characterization of derived equivalence.

\begin{corollary}
\label{cor:hochschild-derived-characterization}
Let $A:=S_n(c,R)$ and $B:=S_m(d,R)$. Then the following conditions are
equivalent:
\begin{enumerate}[label=\textup{(\arabic*)}]
\item The algebra--module pairs $(Z(A),HH_0(A))$ and $(Z(B),HH_0(B))$
are isomorphic.
\item The matrices $c$ and $d$ are D-equivalent.
\item The algebras $A$ and $B$ are derived equivalent as $R$-algebras.
\end{enumerate}
Consequently, the center, the T-type $T_R$, and the support $\tau$-tilting
poset are derived invariants within this class.
\end{corollary}

\begin{proof}
By Proposition~\ref{prop:center} and Theorem~\ref{thm:hh-center-module}, the
primitive central idempotent $e_f$ corresponding to $f\in\Irr(c)$ satisfies
\[
 e_fZ(A)\simeq U_c(f),\ \
 e_fHH_0(A)\simeq
 \bigoplus_{g\in H(P_c(f))}U_c(f)/\rad^gU_c(f),
\]
and the analogous description holds for $B$.
{For $h\in\Irr(d)$, let $e'_h$ denote the corresponding
primitive central idempotent of $Z(B)$.}

For an isomorphism $(\varphi,\psi)$ in \textup{(1)}, $\varphi$ permutes
primitive central idempotents, inducing a bijection
$\sigma\colon\Irr(c)\to\Irr(d)$ and isomorphisms
$\varphi_f\colon U_c(f)\to U_d(\sigma(f))$. Compatibility of $\psi$ with
$\varphi$ then gives an isomorphism from $e_fHH_0(A)$ to
$e'_{\sigma(f)}HH_0(B)$ after transport of scalars along $\varphi_f$.
Since $\varphi_f$ preserves the radical and all its powers, the cyclic
modules $U_c(f)/\rad^gU_c(f)$ are transported to the corresponding cyclic
modules over $U_d(\sigma(f))$. Each such module is indecomposable and has
Loewy length $g$. Therefore, the Krull--Schmidt theorem gives
$H(P_c(f))=H(P_d(\sigma(f)))$ as multisets. Thus $c\simD d$.

Conversely, the block isomorphisms in \textup{(2)} preserve radical powers
and match the gap multisets. By Theorem~\ref{thm:hh-center-module}, they
therefore induce compatible isomorphisms of the blockwise algebra--module
pairs. Taking products proves \textup{(1)}.

{Derived invariance of the cap product
\cite[Theorem~2.2]{ArmentaKeller} gives \textup{(3)}$\Rightarrow$\textup{(1)},
as in the proof of Corollary~\ref{cor:local-derived-reconstruction}.}
Finally, \textup{(2)}$\Rightarrow$\textup{(3)} is Li--Xi's classification
theorem \cite[Theorem~1.1]{LiXiD}.
\end{proof}
\begin{remark}\label{remark:4.8}
The module structure in $(Z(A),HH_0(A))$ is essential: if
$P=\{p_1<\cdots<p_s\}$ and $r=p_s$, then the underlying $R$-vector space
records only
$\dim_R HH_0(\Lambda_f(P))=\deg(f)r$, whereas its module structure over the
center recovers the individual gaps with their multiplicities.

{Over $\mathbb C$, the dimension identity follows from the calculation in
the proof of \cite[Proposition~3.5, p.~711]{Izosimov}. For
$\mathfrak g=\mathfrak{sl}_n(\mathbb C)$ and $a\in\mathfrak g$, write
$\mathfrak g^a$ for the centralizer of $a$ in $\mathfrak g$. The cited
calculation gives
$\dim_{\mathbb C}(\mathfrak g^a/[\mathfrak g^a,\mathfrak g^a])
=\sum_\lambda h(\lambda)-1$, where $h(\lambda)$ is the largest Jordan
block size at the eigenvalue $\lambda$. For $n\geq2$ and
$a=c-\frac{\operatorname{tr}(c)}{n}I_n$, the Lie algebra underlying
$S_n(c,\mathbb C)$ is $\mathfrak g^a\oplus\mathbb C I_n$.
Thus $\dim_{\mathbb C}HH_0(S_n(c,\mathbb C))
=\sum_\lambda h(\lambda)=\deg\mu_c$; the case $n=1$ is immediate.
Theorem~\ref{thm:hh-center-module} refines this identity to the full
center--module structure over an arbitrary field.}
\end{remark}

\begin{proof}[\bf Proof of Theorem~\ref{thm:intro-hochschild}]
The center--module formula and the blockwise recovery assertion follow from
Theorem~\ref{thm:hh-center-module} and
Corollary~\ref{cor:hochschild-derived-characterization}, respectively.
Taking dimensions in the former and using
Corollary~\ref{cor:center-recovers} gives
\[
 \dim_R HH_0(S_n(c,R))
 =\sum_{f\in\Irr(c)}\deg(f)r_c(f)
 =\dim_R Z(S_n(c,R)).
\]
\end{proof}
\subsection{Comparison of the reconstruction relations}
\label{subsec:comparison}

We compare the reconstruction invariants with Li--Xi's matrix equivalences.

For a nonempty finite set
$P=\{p_1<p_2<\cdots<p_s\}\subseteq \mathbb{Z}_{>0}$, define
$J(P):=\{p_s\}\cup\{p_s-p_i\mid 1\le i<s\}$. Unlike the gap multiset
$H(P)$, $J(P)$ is a set. These operations agree with
\cite[Section~3.1]{LiXiD}, rewritten in increasing order. Thus
$|H(P)|=|J(P)|=|P|$, {$J(J(P))=P$}, and $H(J(P))=H(P)$.
\paragraph{The Li--Xi equivalences.}
In \cite[Definition~3.1]{LiXiD}, {Li--Xi index the block
corresponding to $f$ by the maximal elementary divisor $f^{r_c(f)}$;
their power-index set attached to this divisor is precisely $P_c(f)$.
Moreover, $U_c(f)=R[x]/(f^{r_c(f)})$.}

Let $c\in M_n(R)$ and $d\in M_m(R)$. Then:
\begin{enumerate}[label=\textup{(\arabic*)}]
\item $c\simM d$ if there is a bijection
$\sigma\colon\Irr(c)\to\Irr(d)$ such that
$U_c(f)\simeq U_d(\sigma(f))$ and
$P_c(f)=P_d(\sigma(f))$ for every $f\in\Irr(c)$.
\item $c\simAD d$ if there is a bijection
$\sigma\colon\Irr(c)\to\Irr(d)$ such that, for every
$f\in\Irr(c)$, we have $U_c(f)\simeq U_d(\sigma(f))$ and either
$P_c(f)=P_d(\sigma(f))$ or
$P_c(f)=J(P_d(\sigma(f)))$.
\item The notation $c\simD d$ refers to the D-equivalence relation defined
before Corollary~\ref{cor:hochschild-derived-characterization}.
\end{enumerate}
In this order, these are the relations of
\cite[Definition~3.1(1), (3), and (2)]{LiXiD}. They characterize
Morita equivalence, almost $\nu$-stable derived
equivalence, and derived equivalence, respectively
\cite[Theorem~1.1]{LiXiD}.

\begin{definition}
\label{def:center-equivalences}
Let $c\in M_n(R)$ and $d\in M_m(R)$.
\begin{enumerate}[label=\textup{(\arabic*)}]
\item The matrices $c$ and $d$ are \emph{blockwise TZ-equivalent}, written
$c\simbTZ d$, if there is a bijection
$\sigma\colon\Irr(c)\to\Irr(d)$ such that
$U_c(f)\simeq U_d(\sigma(f))$ and
{$\kappa_c(f)=\kappa_d(\sigma(f))$} for every $f\in\Irr(c)$.
\item They are \emph{TZ-equivalent}, written $c\simTZ d$, if their support
$\tau$-tilting posets are isomorphic and their centers are isomorphic as
$R$-algebras.
\end{enumerate}
\end{definition}

Thus bTZ-equivalence matches the center and support $\tau$-tilting data
block by block, while TZ-equivalence records them separately.
Recall that $c\simT d$ means $T_R(c)=T_R(d)$.

\begin{proposition}\label{prop:blockwise}
The matrices \(c\) and \(d\) are bTZ-equivalent if and only if there is a
bijection between the blocks of \(S_n(c,R)\) and \(S_m(d,R)\) such that
corresponding blocks have isomorphic centers as \(R\)-algebras and
isomorphic support \(\tau\)-tilting posets.
\end{proposition}

\begin{proof}
By Proposition~\ref{prop:center}, the center of the primary factor indexed
by \(f\) is the local algebra \(U_c(f)\). Hence the primary factors in \eqref{eq:primary-product} are exactly the
blocks of the centralizer matrix algebra. Moreover, Theorem~\ref{thm:main-recover} shows that the support \(\tau\)-tilting poset of
this block is {\(\Weak(\Sigma_{\kappa_c(f)+1})\)} and determines
{\(\kappa_c(f)\)}. Thus the stated conditions are equivalent to a bijection preserving the
pairs {\((U_c(f),\kappa_c(f))\)}, as required.
\end{proof}

For a nonempty finite set
$P=\{p_1<\cdots<p_s\}\subseteq\mathbb Z_{>0}$, {with the
successive gaps $g_i$ fixed above, define}
$G(P):=\bigoplus_{i=1}^{s}\mathbb Z/g_i\mathbb Z$.
{We use the convention $\mathbb Z/1\mathbb Z=0$.}
For $\Lambda_f(P)$, let $e_i$ be the idempotent corresponding to the
summand $R[x]/(f^{p_i})$, write $P_i:=\Lambda_f(P)e_i$, and let
$S_i:=P_i/\rad P_i$. We define the Cartan matrix $C(P)$ by
$C(P)_{ij}:=[P_i:S_j]$, where $[P_i:S_j]$ denotes the composition
multiplicity, and let
{$\Dsg(\Lambda_f(P)):=
D^{\mathrm b}(\Lambda_f(P)\text{-}\mathrm{mod})/
\Kb(\proj\Lambda_f(P))$} be its singularity
category.

{The Cartan matrix and determinant below are given over arbitrary
fields in \cite[Lemma~2.8]{ChenXi}; see also \cite{DPS12} and
\cite[Lemma~2.18]{LiXiD}. We include the calculation in our notation.
The formula for $K_0(\Dsg(\Lambda_f(P)))$ follows from the standard
cokernel description.}

\begin{proposition}\label{prop:cartan-gap}
Let $f\in R[x]$ be monic and irreducible. Then
\[
 C(P)_{ij}=\min\{p_i,p_j\},\ \
 C(P)={\mathsf L}\operatorname{diag}(g_1,\ldots,g_s)
       {\mathsf L}^{\mathsf T},
\]
where {$\mathsf L_{ij}=1$ if $j\le i$ and
$\mathsf L_{ij}=0$ otherwise}. Consequently,
\[
 \det C(P)=\prod_{i=1}^{s}g_i,\ \
 {K_0(\Dsg(\Lambda_f(P)))\simeq\coker C(P)
 \simeq\bigoplus_{i=1}^{s}\mathbb Z/g_i\mathbb Z}=G(P).
\]
\end{proposition}

\begin{proof}
{Let $M_i:=R[x]/(f^{p_i})$, and write $\ell_{R[x]}$ for
composition length over $R[x]$.} Since $e_j(-)$
is exact, applying it to a composition series of $P_i$ gives
$[P_i:S_j]={\ell_{R[x]}}(e_jP_i)$. Indeed, $e_jS_k=0$ for
$k\ne j$, while {$e_jS_j\simeq K_f$ has $R[x]$-length one}.
Since
$e_jP_i=e_j\Lambda_f(P)e_i\simeq
{\Hom_{R[x]}(M_i,M_j)}$, it follows that
$[P_i:S_j]={\ell_{R[x]}\Hom_{R[x]}(M_i,M_j)}
=\min\{p_i,p_j\}$.
Since $p_i=g_1+\cdots+g_i$, the stated factorization follows, and
{$\mathsf L$} is
unimodular. Therefore {$\coker C(P)\simeq\bigoplus_i\mathbb Z/g_i\mathbb Z$}.
Moreover, the homomorphism
{$K_0(\Kb(\proj\Lambda_f(P)))\to
K_0(D^{\mathrm b}(\Lambda_f(P)\text{-}\mathrm{mod}))$} is represented by $C(P)$ in the bases of indecomposable projectives and
simple modules. Thus the exact sequence on Grothendieck groups for the
Verdier quotient gives {$K_0(\Dsg(\Lambda_f(P)))\simeq\coker C(P)$}.
\end{proof}

The group $G(P)$ does not determine the singularity category; for a complete
description of singularity categories and singular equivalences of
centralizer matrix algebras, see Chen--Xi
\cite[Theorems~4.7 and~4.9]{ChenXi}.
Since $G(P)$ does not record gaps equal to one, we retain $|P|$ in the
next definition.

\begin{definition}\label{def:gTZ}
The matrices $c$ and $d$ are
\emph{Grothendieck-enhanced blockwise TZ-equivalent}, abbreviated
\emph{gTZ-equivalent}, written
$c\simgTZ d$, if there is a bijection $\sigma\colon\Irr(c)\to\Irr(d)$
such that, for every $f\in\Irr(c)$,
$U_c(f)\simeq U_d(\sigma(f))$,
{$\kappa_c(f)=\kappa_d(\sigma(f))$}, and
$G(P_c(f))\simeq G(P_d(\sigma(f)))$ as abelian groups.
\end{definition}

\begin{theorem}
\label{thm:strict-comparison}
Over every field $R$, there is a strict implication chain
\[
\mathrm{D}\Rightarrow\mathrm{gTZ}\Rightarrow
\mathrm{bTZ}\Rightarrow\mathrm{TZ}\Rightarrow\mathrm{T}.
\]
\end{theorem}

\begin{proof}
Since the gap multiset $H(P)$ determines both $|P|$ and $G(P)$, it follows
that $\mathrm{D}\Rightarrow\mathrm{gTZ}$. The definition then gives
$\mathrm{gTZ}\Rightarrow\mathrm{bTZ}$.

Taking products of the corresponding block centers and support
$\tau$-tilting posets in Proposition~\ref{prop:blockwise} gives
$\mathrm{bTZ}\Rightarrow\mathrm{TZ}$. Finally,
$\mathrm{TZ}\Rightarrow\mathrm{T}$ follows from
Theorem~\ref{thm:main-recover}.

For strictness, write $J_a(\lambda)$ for the Jordan
block of size $a$ with eigenvalue $\lambda$.

First, let $c_P:=J_3(0)\oplus J_7(0)\oplus J_{21}(0)$ and
$c_Q:=J_2(0)\oplus J_9(0)\oplus J_{21}(0)$.
Their exponent sets $P=\{3,7,21\}$ and $Q=\{2,9,21\}$ have the same
cardinality and
$G(P)\simeq\mathbb Z/2\mathbb Z\oplus\mathbb Z/84\mathbb Z
\simeq G(Q)$, whereas
$H(P)=\ms{3,4,14}\ne\ms{2,7,12}=H(Q)$.
Thus $c_P$ and $c_Q$ are gTZ-equivalent but not D-equivalent. Second, the
nilpotent matrices with exponent sets $P=\{1,4\}$ and $Q=\{2,4\}$ have the
same blockwise data $(R[t]/(t^4),2)$, whereas
$G(P)\simeq\mathbb Z/3\mathbb Z$ and
$G(Q)\simeq(\mathbb Z/2\mathbb Z)^2$.
Hence bTZ-equivalence does not imply gTZ-equivalence.

For the third separation, let
\[
 \begin{aligned}
 c&=J_1(0)\oplus J_5(0)\oplus J_2(1),\\
 d&=J_1(0)\oplus J_2(0)\oplus J_5(1).
 \end{aligned}
\]
Both posets are $\Weak(\Sigma_3)\times\Weak(\Sigma_2)$, and both centers
are $R[t]/(t^5)\times R[t]/(t^2)$. Thus $c\simTZ d$. {However,
their blockwise data are, respectively,
\[
 \ms{(R[t]/(t^5),2),(R[t]/(t^2),1)}
 \quad\text{and}\quad
 \ms{(R[t]/(t^2),2),(R[t]/(t^5),1)}.
\]
Hence $c\not\simbTZ d$.} Finally, the exponent sets $\{1,2\}$ and $\{1,3\}$
have the same T-type, but their centers have Loewy lengths $2$ and $3$.
Therefore T-equivalence does not imply TZ-equivalence. All four examples
work over every field.
\end{proof}

\begin{remark}\label{rem:full-strict-chain}
{The Li--Xi relations add the known strict implications
\[
\mathrm{M}\Rightarrow\mathrm{AD}\Rightarrow\mathrm{D}
\]
to the left of the chain in Theorem~\ref{thm:strict-comparison}.
The first implication follows from the definitions, and the second
from $H(J(P))=H(P)$. The nilpotent matrices with exponent sets
$\{1,3,5\}$ and $\{2,4,5\}$ are $\mathrm{AD}$-equivalent but not
$\mathrm{M}$-equivalent \cite[Section~3.1, Example~(2)]{LiXiD}.
Strictness of the second implication is also established in
\cite[Example~5.3]{LiXiD}. In our notation, another example is
$P=\{1,3,6\}$ and $Q=\{1,4,6\}$: these sets have the same gap multiset
$\ms{1,2,3}$, but $P\ne Q$, $J(Q)=\{2,5,6\}\ne P$, and
$J(P)=\{3,5,6\}\ne Q$. Thus the corresponding matrices are
$\mathrm{D}$-equivalent but not $\mathrm{AD}$-equivalent.}
\end{remark}

The following dimension formulas are known; we record which reconstruction
data determine them.

\begin{proposition}
\label{prop:homological-consequences}
Let $\Lambda_f(P)$ be the basic Morita representative associated with a
primary block, where $r:=\max P$. Here $\domdim$ and $\gldim$ denote
dominant and global dimension, respectively. Then
\[
 \domdim\Lambda_f(P)=
 \begin{cases}
 \infty,&P=\{r\};\\
 2,&\text{otherwise},
 \end{cases}
 \ \
 \gldim\Lambda_f(P)=
 \begin{cases}
 0,&P=\{1\};\\
 2,&P=\{1,\ldots,r\}\text{ and }r\ge2;\\
 \infty,&\text{otherwise}.
 \end{cases}
\]
Moreover, $\gldim\Lambda_f(P)<\infty$ if and only if
$\det C(P)=1$. Consequently, for centralizer matrix algebras, the multiset
$T_R(c)$ determines dominant dimension, whereas the pair
$(T_R(c),Z(S_n(c,R)))$ determines global dimension; $T_R(c)$ alone does
not.
\end{proposition}

\begin{proof}
The dominant-dimension formula is \cite[Lemma~4.8]{LiXiD}.
For perfect fields, the finite-global-dimension criterion and the value
$2$ in the non-semisimple finite-global-dimension case follow from
\cite[Theorem~9.2]{DPS12}. The remaining case $P=\{1\}$ is semisimple
and therefore has global dimension $0$. The complete formula over an
arbitrary field is \cite[Corollary~4.5]{ChenXi}. By
Proposition~\ref{prop:cartan-gap},
$\det C(P)=\prod_{g\in H(P)}g$.
Hence the determinant equals $1$ if and only
if $P=\{1,\ldots,r\}$; compare also \cite[Lemma~6.4]{ChenXi}.

For finite direct products, dominant dimension is the minimum of the
dominant dimensions of the factors, and global dimension is their maximum.
The multiset $T_R(c)$ determines whether every primary block has only one
exponent, and hence determines dominant dimension. For every $f$, we have
$\kappa_c(f)\le r_c(f)$, with equality if and only if
$P_c(f)=\{1,\ldots,r_c(f)\}$. Since each difference
$r_c(f)-\kappa_c(f)$ is nonnegative, all primary blocks have finite global
dimension if and only if
$\sum_{f\in\Irr(c)}\kappa_c(f)=
\sum_{f\in\Irr(c)}r_c(f)$.
The T-type determines the first sum, and the Loewy lengths of the center
factors determine the second. When finite, the global dimension is zero
if all these Loewy lengths are one, and two otherwise. Thus
$(T_R(c),Z(S_n(c,R)))$ determines global dimension. The exponent sets
$\{1,2\}$ and $\{1,3\}$ have the same T-type but global dimensions $2$
and $\infty$, respectively, so $T_R(c)$ alone does not determine it.
\end{proof}

\section{Morita-string and Morita-gentle centralizer matrix algebras}\label{sec:string}

{In this section, we  will  prove Theorem~\ref{thm:collapse}: bTZ recovers the Morita class in the
Morita-string class, and TZ does so in the Morita-gentle class. The local
string classification follows from Marczinzik's theorem on monomial
algebras of dominant dimension at least two \cite{MarczinzikMonomial};
we give a direct proof and obtain the gentle criterion by imposing
quadratic relations.}

\subsection{The split obstruction and the local criterion}
\label{subsec:local-string}

{We use the bound-quiver notions of string and gentle algebras over $R$;
see \cite{BR,AssemSkowronski} and the modern gentle convention in
\cite[pp.~206--207]{AssemSurface}.}
A string $R$-algebra is an algebra of the form $RQ/I$, where $Q$ is a finite quiver with vertex set $Q_0$,
$I$ is an admissible monomial ideal, each vertex is the source and the target of at most two arrows,
and each arrow has at most one permitted continuation on either side.
It is gentle if, in addition, $I$ is generated by paths of length two and each arrow has at most one forbidden continuation on either side.
{Here a continuation is permitted if the corresponding length-two path is
not in $I$, and forbidden otherwise. These algebras are split basic, with
semisimple quotient $R^{Q_0}$. We allow disconnected quivers, loops and
oriented cycles; no triangularity assumption is imposed.}

\begin{definition}\label{def:string-gentle-scope}
A finite-dimensional $R$-algebra is \emph{Morita-string} if it is $R$-linearly Morita equivalent to a string $R$-algebra.
It is \emph{Morita-gentle} if it is $R$-linearly Morita equivalent to a gentle $R$-algebra.
\end{definition}

A \emph{Gabriel presentation} of a split basic algebra $A$ is an
isomorphism $A\simeq RQ/I$ with $Q$ the Gabriel quiver of $A$ and $I$ an
admissible ideal; see \cite[Chap.~II]{ASS2006}. It is \emph{monomial} if
$I$ is monomial. The following lemma allows us to transfer such
presentations under Morita equivalence.

\begin{lemma}
\label{lem:morita-monomial-transfer}
Let $A$ and $B$ be finite-dimensional basic $R$-algebras. If they are
$R$-linearly Morita equivalent, then $A\simeq B$ as $R$-algebras.
Consequently, if $B$ has a monomial Gabriel presentation, then so does $A$.
\end{lemma}

\begin{proof}
An $R$-linear equivalence
{$F:A\text{-}\mathrm{mod}\to B\text{-}\mathrm{mod}$} induces a
bijection between the indecomposable projective modules.  Since $A$ and
$B$ are basic, their left regular modules contain each indecomposable
projective exactly once. Hence $F({}_AA)\simeq{}_BB$. Since $F$ is
$R$-linear and fully faithful, it follows that
$
 A^{\mathrm{op}}=\End_A({}_AA)
 \simeq\End_B(F({}_AA))
 \simeq\End_B({}_BB)=B^{\mathrm{op}}
$
as $R$-algebras. Taking opposite algebras gives $A\simeq B$.
Transporting a monomial Gabriel presentation of $B$ along this isomorphism
proves the last assertion.
\end{proof}

We classify the basic primary models $\Lambda_f(P)$ for monic irreducible
$f$.

Over a perfect field, the Morita-string condition is strictly more
restrictive than representation-finiteness. For example, if $r\ge4$,
then $P=\{1,r\}$ gives a representation-finite block that is not
Morita-string (see \cite[Lemma~3.3]{LiXiD} and
\cite[Theorem~2.1(i)]{DrozdMazorchuk}).
{Over algebraically closed fields, Chan--Marczinzik classify
representation-finite gendo-symmetric biserial algebras by Brauer-tree
and hook-module data \cite[Theorem~3.9]{ChanMarczinzik}; their
Corollary~3.10 treats the Nakayama case. Here we specialize the monomial
classification of \cite{MarczinzikMonomial}, valid over an arbitrary
field, to primary blocks of centralizer matrix algebras.}

The semisimple quotient first forces $f$ to be linear. We then use the
Gabriel quiver and radical-filtration degrees to exclude the nonmonomial
cases.

\begin{lemma}\label{lem:linear-obstruction}
Let $f\in R[x]$ be monic and irreducible, and let $P=\{p_1<\cdots<p_s\}\subseteq\mathbb Z_{>0}$ be nonempty.
With $K_f$ as fixed above, $\Lambda_f(P)/\rad \Lambda_f(P)\simeq \prod_{i=1}^{s}K_f$.
Consequently, if $\Lambda_f(P)$ is $R$-linearly Morita equivalent to a split basic $R$-algebra, then $K_f\simeq R$ as $R$-algebras, i.e., $f$ is linear.
\end{lemma}

\begin{proof}
Let $X_i:=R[x]/(f^{p_i})$, $X:=\bigoplus_iX_i$, $A:=\Lambda_f(P)$, and let $e_i$ correspond to $X_i$.
Each $\End_{R[x]}(X_i)\simeq R[x]/(f^{p_i})$ is local.
Since the $X_i$ are pairwise nonisomorphic, every morphism between distinct summands belongs to the radical of the Krull--Schmidt category $\add X$,
while an endomorphism of $X_i$ belongs to that radical exactly when it is noninvertible. It follows that
\[
 e_j(\rad A)e_i=
 \begin{cases}
  \rad\End_{R[x]}(X_i)=(f)/(f^{p_i}),&i=j,\\[2pt]
  \Hom_{R[x]}(X_i,X_j),&i\ne j.
 \end{cases}
\]
Hence
\[
 A/\rad A\simeq
 \prod_{i=1}^s
 \frac{\End_{R[x]}(X_i)}{\rad\End_{R[x]}(X_i)}
 \simeq K_f^{\,s}.
\]
Every simple $A$-module therefore has endomorphism division ring $K_f$.
An $R$-linear Morita equivalence preserves these rings, which are $R$ for
a split basic algebra. Hence $K_f\simeq R$ as $R$-algebras and
$\deg f=\dim_R K_f=1$.
\end{proof}

It remains to consider linear $f$. For $p\geq1$, set
$R_p:=R[t]/(t^p)$ and $N_p:=\End_{R[t]}(R_p\oplus R_{p+1})$.
The following lemma describes the two-point string family.

\begin{lemma}\label{lem:two-point}
For $p\geq1$, let $Q^{(2)}$ be the two-vertex quiver
\[
 Q^{(2)}:\qquad
 1\underset{b}{\overset{a}{\rightleftarrows}}2.
\]
Then $N_p\simeq RQ^{(2)}/((ba)^p)$.
Under this isomorphism, $a$ is the canonical inclusion
$R_p\to R_{p+1}$ given by multiplication by $t$, and $b$ is the
canonical projection $R_{p+1}\to R_p$. Moreover, $N_p$ is a connected
representation-finite Nakayama algebra that is also a string algebra,
$Z(N_p)\simeq R[z]/(z^{p+1})$ for $z:=ba+ab$, and
$Z(R_p)=R_p$.
\end{lemma}

\begin{proof}
Write $M_1:=R_p$ and $M_2:=R_{p+1}$. Let
$a([g]_{t^p}):=[tg]_{t^{p+1}}$ and
$b([g]_{t^{p+1}}):=[g]_{t^p}$, where $[g]_{t^q}$ denotes the residue
class of $g$ modulo $(t^q)$. Then $ba=\mu_t|_{M_1}$ and
$ab=\mu_t|_{M_2}$, where $\mu_t$ denotes multiplication by $t$.
Thus $(ba)^p=0$, $(ab)^p\ne0$, and
$(ab)^{p+1}=a(ba)^pb=0$.
The maps $a,b$ and the vertex idempotents induce a homomorphism
\(
 \Phi\colon RQ^{(2)}/((ba)^p)\longrightarrow N_p.
\)
{Here and below, $(ba)^0:=\varepsilon_1$ and
$(ab)^0:=\varepsilon_2$, the
trivial paths at their respective vertices.}
{The following nonzero paths form a basis:}
\[
\begin{aligned}
 &
 \{(ba)^r\mid0\le r<p\}
 \cup\{(ab)^r\mid0\le r\le p\}\\
 &{}\cup\{a(ba)^r\mid0\le r<p\}
 \cup\{(ba)^rb\mid0\le r<p\}.
\end{aligned}
\]
{Under $\Phi$, the four displayed parts are the standard $t$-power bases of}
$\End(M_1)$, $\End(M_2)$, $\Hom(M_1,M_2)$ and $\Hom(M_2,M_1)$,
respectively. Thus $\Phi$ is an isomorphism and $\dim_R N_p=4p+1$.
The monomial presentation has one incoming and one outgoing arrow at
each vertex, so $N_p$ is a connected string algebra. Its indecomposable
projectives are uniserial of lengths $2p$ and $2p+1$, making it a
representation-finite Nakayama algebra.

Since the summands of $z$ lie in orthogonal corners,
$z^r=(ba)^r+(ab)^r$ for $0\le r<p$, while
$z^p=(ab)^p\ne0$ and $z^{p+1}=0$. If $\xi\in Z(N_p)$, then commutation
with the vertex idempotents gives
$\xi=\sum_{r=0}^{p-1}\lambda_r(ba)^r+
\sum_{r=0}^{p}\mu_r(ab)^r$.
The equations $\xi a=a\xi$ and $\xi b=b\xi$ imply
$\lambda_r=\mu_r$ for $0\le r<p$, while $\mu_p$ is free.
Consequently, $Z(N_p)=\bigoplus_{r=0}^{p}Rz^r
\simeq R[z]/(z^{p+1})$.
The assertion for $R_p$ follows from commutativity.
\end{proof}

Let $A$ be a finite-dimensional split basic $R$-algebra with Gabriel
quiver $Q$. Choose a complete set of primitive orthogonal idempotents
$\{e_i\}_{i\in Q_0}$ lifting the vertex idempotents of $A/\rad A$.
For each $i,j\in Q_0$, choose a basis of
$e_j\rad A e_i/e_j\rad^2 A e_i$ indexed by the arrows
$\alpha\colon i\to j$, with lifts $\widetilde\alpha\in e_j\rad A e_i$.
{These choices determine a unique surjective $R$-algebra homomorphism
$\phi\colon RQ\twoheadrightarrow A$ with
$\phi(\varepsilon_i)=e_i$ and $\phi(\alpha)=\widetilde\alpha$, where
$\varepsilon_i$ is the trivial path at $i$; see
\cite[Chap.~II, Section~3]{ASS2006}. The construction uses the split
basic hypothesis and does not require $R$ to be algebraically closed.}

\begin{lemma}\label{lem:obstruction}
Suppose that for every such choice of $\phi$, there exist distinct paths
$w_1,\dots,w_m$ in $Q$ such that $\phi(w_i)\ne 0 $ for every $i$,
but $\phi(w_1),\dots,\phi(w_m)$ are linearly dependent.
Then $A$ has no monomial Gabriel presentation.
\end{lemma}

\begin{proof}
A monomial Gabriel presentation would give one of the maps $\phi$ in
the statement with monomial kernel $I$. The paths outside $I$ form an
$R$-basis of $RQ/I$, so distinct paths with nonzero images under $\phi$
would be linearly independent, a contradiction.
\end{proof}

For the rest of this subsection, write $M(r):=R[t]/(t^r)$ for $r\ge1$. If
$P=\{p_1<\cdots<p_s\}$, set
$\Lambda(P):=\End_{R[t]}(\bigoplus_{i=1}^s M(p_i))$.

For each $r$, equip $M(r)$ with its radical (equivalently, $t$-adic)
filtration
\[
 M(r)=t^0M(r)\supset tM(r)\supset\cdots\supset
 t^{r-1}M(r)\supset t^rM(r)=0.
\]
For a nonzero map
$\varphi\in\Hom_{R[t]}(M(a),M(b))$, define its
\emph{radical-filtration degree} by
\[
 \fdeg_t(\varphi):=
 \max\bigl\{d\in\{0,\ldots,b-1\}\mid
       \operatorname{Im}\varphi\subseteq t^dM(b)\bigr\},
\]
and set $\fdeg_t(0):=\infty$. Equivalently, $\fdeg_t(\varphi)=d$ if and only if $\varphi(1)\in t^dM(b)\setminus t^{d+1}M(b)$.
We use the conventions $\infty>n$ and $n+\infty=\infty+n=\infty+\infty=\infty$ for every integer $n\ge0$.

For $1\leq a<b$, let $\iota_{a,b}\colon M(a)\to M(b)$ and $\pi_{b,a}\colon M(b)\to M(a)$ be given by $\iota_{a,b}(1):=t^{b-a}$ and $\pi_{b,a}(1):=1$.
These are the standard inclusion and projection. A \emph{backtrack} is the closed path of length two that follows an arrow between adjacent vertices and immediately returns along the opposite arrow.

The next lemma computes the Gabriel quiver and the filtration degrees
needed for Lemma~\ref{lem:obstruction}.

\begin{lemma}\label{lem:quiver-rigidity}
Let $P=\{p_1<\cdots<p_s\}$ and let $Q$ be the Gabriel quiver of
$\Lambda(P)$. The arrows of $Q$ are as follows:
\begin{enumerate}[label=\textup{(\arabic*)}]
\item for consecutive $p<q$ in $P$, the standard inclusion $M(p)\to M(q)$ and the standard projection $M(q)\to M(p)$;
\item a loop at $M(r)$ exactly when $r\ge 2$ and neither $r-1$ nor $r+1$ belongs to $P$.
\end{enumerate}
Each arrow space is one-dimensional. Moreover:
\begin{enumerate}[label=\textup{(\alph*)}]
\item if $p<q$ are consecutive in $P$, then a representative of the inclusion arrow $M(p)\to M(q)$ has radical-filtration degree $q-p$, a representative of the projection arrow $M(q)\to M(p)$ has radical-filtration degree $0$, and their closed two-step backtrack at $M(q)$ has radical-filtration degree $q-p$;
\item a loop at $M(q)$ has radical-filtration degree one; if $q+1\in P$, the backtrack
$M(q)\to M(q+1)\to M(q)$ has radical-filtration degree one for $q\ge2$, while for $q=1$
this backtrack is zero;
\item if $p,p+1,p+2\in P$, both backtracks at $M(p+1)$ have radical-filtration degree one.
\end{enumerate}
\end{lemma}

\begin{proof}
\smallskip
\noindent\emph{Step 1.}
Evaluation at $1$ identifies a homomorphism with an element of $M(b)$
annihilated by $t^a$.  Hence
\[
\begin{aligned}
 \Hom_{R[t]}(M(a),M(b))
 &=t^{(b-a)_+}R[t]/(t^b)\\
 &=\operatorname{span}_R
   \{t^{(b-a)_+},t^{(b-a)_++1},\dots,t^{b-1}\}.
\end{aligned}
\]
In particular,
$\dim_R \Hom_{R[t]}(M(a),M(b))=\min\{a,b\}$.
If $h\geq1$ and
\(
 M(a)\xrightarrow{\varphi}M(h)
 \xrightarrow{\psi}M(b)
\)
is any factorization and
$\alpha:=\fdeg_t(\varphi)$ and $\beta:=\fdeg_t(\psi)$ are finite, write
\[
 \varphi(1)=t^\alpha u(t),\qquad
 \psi(1)=t^\beta v(t),
\]
where $u(0),v(0)\ne0$. Then
\[
 \fdeg_t(\psi\varphi)=
 \begin{cases}
  \alpha+\beta,&\alpha+\beta<b,\\
  \infty,&\alpha+\beta\ge b.
 \end{cases}
\]
Thus, including zero maps,
\[
 \fdeg_t(\psi\varphi)\ge
 \fdeg_t(\varphi)+\fdeg_t(\psi).
\]
Every map $M(a)\to M(h)$ has filtration degree at least $(h-a)_+$, and
every map $M(h)\to M(b)$ has filtration degree at least $(b-h)_+$.
Consequently,
\(
 \fdeg_t(\psi\varphi)\ge
 (h-a)_+ +(b-h)_+.
\)
For the corresponding standard inclusion--projection factorization, the
composite is multiplication by
\(
 t^{(h-a)_+ +(b-h)_+}.
\)
The bound is therefore attained whenever this standard composite is
nonzero.

\smallskip
\noindent\emph{Step 2.}
Let $\mathfrak r$ denote the radical of the additive category generated by
$\{M(p_i)\}_{i=1}^s$, and let $e_a$ be the idempotent of $\Lambda(P)$
corresponding to $M(a)$. Since the modules $M(p_i)$ are pairwise
nonisomorphic and have local endomorphism rings, the argument used in
Lemma~\ref{lem:linear-obstruction} gives
$e_b(\rad\Lambda(P))e_a=\mathfrak r(M(a),M(b))$ and
$e_b(\rad^2\Lambda(P))e_a=\mathfrak r^2(M(a),M(b))$.
Thus the arrow space from $M(a)$ to $M(b)$ is
$\mathfrak r(M(a),M(b))/\mathfrak r^2(M(a),M(b))$.  Its numerator is
\[
 \mathfrak r(M(a),M(b))=
 \begin{cases}
  \Hom_{R[t]}(M(a),M(b)),&a\ne b,\\
  tR[t]/(t^a),&a=b.
 \end{cases}
\]
Moreover,
$\mathfrak r^2(M(a),M(b))=
{\sum_{h\in P}\mathfrak r(M(h),M(b))\,
\mathfrak r(M(a),M(h))}$.
If $a<h<b$ all belong to $P$, then
\(
 \iota_{a,b}=\iota_{h,b}\iota_{a,h},
  \
 \pi_{b,a}=\pi_{h,a}\pi_{b,h},
\)
so the standard maps lie in $\mathfrak r^2$.  Every higher $t$-multiple
factors through a radical endomorphism at an endpoint.  Therefore there are
no arrows between nonconsecutive vertices.

Now let $p<q$ be consecutive in $P$. If
$h\in P\setminus\{p,q\}$, then $h<p$ or $h>q$. The estimate in Step~1
shows that a factorization through $M(h)$ has strictly larger filtration
degree than $q-p$ for the inclusion and than $0$ for the projection,
whereas $\fdeg_t(\iota_{p,q})=q-p$ and
$\fdeg_t(\pi_{q,p})=0$.
An endpoint factorization contains a radical endomorphism and also
strictly increases the filtration degree.
{For $a,b\in P$ and
$\theta\in\mathfrak r(M(a),M(b))$, write $\overline\theta$ for its
class in
$\mathfrak r(M(a),M(b))/\mathfrak r^2(M(a),M(b))$.}
Hence
\[
 \frac{\mathfrak r(M(p),M(q))}
      {\mathfrak r^2(M(p),M(q))}
 =R\overline\iota_{p,q}, \quad
 \frac{\mathfrak r(M(q),M(p))}
      {\mathfrak r^2(M(q),M(p))}
 =R\overline\pi_{q,p}.
\]

\smallskip
\noindent\emph{Step 3.}
At $M(r)$, $\rad\End_{R[t]}(M(r))=tR[t]/(t^r)$, and every power $t^j$ with $j\ge2$ already factors through a radical
endomorphism at $M(r)$.  Thus only multiplication by $t$ can define a loop.
For $h\in P\setminus\{r\}$, the standard backtrack through $M(h)$ sends
$1$ to $t^{|r-h|}$. Thus, for $r\ge2$,
$t\,\mathrm{id}_{M(r)}\in\mathfrak r^2(M(r),M(r))$ if and only if
$r-1\in P$ or $r+1\in P$. For $r=1$, multiplication by $t$ is zero.
This proves \textup{(2)} and the one-dimensionality of each loop space.

\smallskip
\noindent\emph{Step 4.}
Because every arrow space is one-dimensional, each chosen arrow
representative has the form
$\widetilde\theta_i=c_i\theta_i+\eta_i$, where
$c_i\in R^\times$, $\eta_i\in\mathfrak r^2$, and $\theta_i$ is the
corresponding standard map. Write $\eta_i$ as a sum of radical
factorizations. Each term factors through an
intermediate summand or a radical endomorphism at an endpoint. By
Steps~2 and~3, its image lies in $t^{\fdeg_t(\theta_i)+1}$ times the
target module. The same holds for their sum, so
$\fdeg_t(\eta_i)>\fdeg_t(\theta_i)$ whenever $\eta_i\ne0$.
Consider a path occurring in \textup{(a)}--\textup{(c)} whose standard
composite is nonzero. Let $M(b)$ be its target and let $d<b$ be its
filtration degree.
By Step~1, every term containing some $\eta_i$ has image in
$t^{d+1}M(b)$. Modulo this submodule, the chosen composite equals
$\prod_i c_i$ times the standard composite, whose class in
$t^dM(b)/t^{d+1}M(b)$ is nonzero. Since $\prod_i c_i\ne0$, its filtration
degree remains $d$, as asserted in \textup{(a)}--\textup{(c)}.
For $q=1$, Step~1 gives filtration degree at least one for the backtrack
on $M(1)$, so it is zero.
\end{proof}

{To apply \cite{MarczinzikMonomial}, set $n:=\max P$,
$B:=R[t]/(t^n)$, and $X:=\bigoplus_{p\in P}B/(t^p)$.
Since $B$ is symmetric and is a summand of $X$, the module $X$ is a
generator-cogenerator, so $\Lambda(P)=\End_B(X)$ has dominant dimension
at least two. Marczinzik's theorem implies that a monomial $\Lambda(P)$
is Nakayama; Lemma~\ref{lem:quiver-rigidity} then forces $P$ to be a
singleton or a consecutive pair. The next proposition gives a direct
proof of this restriction for every choice of a Gabriel presentation.}

\begin{proposition}\label{prop:monomial-obstruction}
Let $P=\{p_1<\cdots<p_s\}$ with $s\ge 2$.
\begin{enumerate}[label=\textup{(\arabic*)}]
\item If consecutive $p<q$ in $P$ satisfy $q-p>1$, then $\Lambda(P)$ has no monomial Gabriel presentation.
\item If $p,p+1,p+2\in P$, then $\Lambda(P)$ has no monomial Gabriel presentation.
\end{enumerate}
\end{proposition}

\begin{proof}
Let $\{e_i\}$ be the standard primitive idempotents and let $\{e_i'\}$ be
any other complete set of primitive orthogonal idempotents. After
reindexing, $e_i'-e_i\in\rad\Lambda(P)$ for every $i$. Define
$w:=\sum_i e_i'e_i$. Then $w\equiv 1\pmod{\rad\Lambda(P)}$, so $w$ is a
unit, and $e_i'w=we_i$. Thus $e_i'=we_iw^{-1}$. Since conjugation preserves
nonvanishing and linear dependence, it suffices to use the standard
idempotents. Fix arbitrary arrow representatives and apply
Lemma~\ref{lem:quiver-rigidity}.

(1) Write $d:=q-p>1$. Let $\rho$ be the image in $\End_{R[t]}(M(q))$ of the backtrack at $M(q)$ through $M(p)$. By Lemma~\ref{lem:quiver-rigidity}(a), $\rho=t^d u(t)$ with
$u(0)\ne 0$. Let $\omega$ be the loop at $M(q)$ if $q+1\notin P$, and otherwise the backtrack at $M(q)$ through $M(q+1)$. Since $p<q$ are consecutive and $q-p>1$, we have $q-1\notin P$ and $q\ge3$. Thus the loop exists in the first case. By Lemma~\ref{lem:quiver-rigidity}(b), $\omega=t v(t)$ with
$v(0)\ne 0$. For $1\le j<q$, write
$\omega^j=t^jv(t)^j=t^j(v(0)^j+t h_j(t))$ for a polynomial $h_j(t)$. Thus the change-of-basis matrix from
$\{t,t^2,\ldots,t^{q-1}\}$ to
$\{\omega,\omega^2,\ldots,\omega^{q-1}\}$ is triangular with nonzero
diagonal entries $v(0)^j$. Hence these powers form a basis of the maximal ideal, $\omega^q=0$, and
\[
\rho=c_d\omega^d+c_{d+1}\omega^{d+1}+\cdots+c_{q-1}\omega^{q-1}
\]
with $c_d\ne 0$ because $\fdeg_t(\rho)=d$. The elements
$\rho,\omega^d,\omega^{d+1},\dots,\omega^{q-1}$ are images of distinct
paths in the Gabriel quiver: $\rho$ uses the two arrows between $M(p)$ and
$M(q)$, while the powers of $\omega$ use only the loop at $M(q)$ or the
two arrows between $M(q)$ and $M(q+1)$. All these images have filtration degree less than $q$, so they are
nonzero and Lemma~\ref{lem:obstruction} applies.

(2) Let $\rho_-$ and $\rho_+$ be the images in $\End_{R[t]}(M(p+1))$ of the backtracks at $M(p+1)$ through $M(p)$ and $M(p+2)$, respectively. By Lemma~\ref{lem:quiver-rigidity}(c),
$\rho_-=t u_-(t)$ and $\rho_+=t u_+(t)$, with
$u_-(0),u_+(0)\ne 0$. Since $\rho_+$ generates the maximal ideal
$(t)/(t^{p+1})$ of $R[t]/(t^{p+1})$ and $\rho_+^{p+1}=0$,
\[
\rho_-=c_1\rho_+ + c_2\rho_+^2+\cdots+c_p\rho_+^p
\]
with $c_1\ne 0$. These are images of distinct paths: the path for
$\rho_-$ passes through $M(p)$, whereas those for $\rho_+^j$ alternate
through $M(p+2)$ and have different lengths. Their filtration degrees
are at most $p$, so all have nonzero images in $R[t]/(t^{p+1})$.
Lemma~\ref{lem:obstruction} again applies.
\end{proof}

{Combining the split and monomial obstructions gives the local string
classification; imposing quadratic relations gives the gentle criterion.}

\begin{theorem}\label{thm:local-string}
Let $f\in R[x]$ be monic and irreducible, and let $P$ be a nonempty finite
set of positive integers. Then $\Lambda_f(P)$ is Morita-string if and only if
\[
 \deg f=1\quad\text{and}\quad
 P=\{p\}\ \text{or}\ P=\{p,p+1\}\quad(p\ge1).
\]
In these cases, $\Lambda_f(P)$ is isomorphic to $R_p$ or $N_p$,
respectively, and is already split basic. Hence it is Morita-gentle if and
only if it is gentle, which occurs precisely for
\[
R_1=R,\ \ R_2=R[\varepsilon]/(\varepsilon^2),\ \
N_1=\End_{R[\varepsilon]/(\varepsilon^2)}
   (R\oplus R[\varepsilon]/(\varepsilon^2)).
\]
\end{theorem}

\begin{proof}
Suppose that $\Lambda_f(P)$ is Morita-string.  By
Lemma~\ref{lem:linear-obstruction}, $\deg f=1$.  Thus $f=x-\lambda$ for
some $\lambda\in R$, and the change of variable $t:=x-\lambda$ gives
$\Lambda_f(P)\simeq
\End_{R[t]}(\bigoplus_{r\in P}M(r))=\Lambda(P)$.
Since
\(
 \Lambda(P)/\rad\Lambda(P)\simeq R^{|P|},
\)
the algebra $\Lambda(P)$ is split basic. Lemma~\ref{lem:morita-monomial-transfer}
therefore transfers the monomial Gabriel presentation of its string
representative to $\Lambda(P)$.

Proposition~\ref{prop:monomial-obstruction} forces consecutive exponents
to differ by one and excludes three consecutive exponents. Hence
$P=\{p\}$ or $P=\{p,p+1\}$. Conversely,
$\Lambda(\{p\})=\End_{R[t]}(M(p))\simeq R_p$ and
$\Lambda(\{p,p+1\})=N_p$.
The algebras $R_p$ are monomial string algebras, and $N_p$ is string by
Lemma~\ref{lem:two-point}.

By Lemma~\ref{lem:morita-monomial-transfer}, these split basic algebras
are Morita-gentle exactly when they are gentle. If $P=\{p\}$, then
$R_1=R$ and $R_p\simeq R[\gamma]/(\gamma^p)$ for $p\ge2$,
where $\gamma$ is the unique loop in the Gabriel quiver. Any representative
$u$ of this loop generates the maximal ideal of $R_p$. Hence
$1,u,\ldots,u^{p-1}$ is an
$R$-basis of $R_p$, while $u^p=0$. Therefore every Gabriel presentation has
kernel $(\gamma^p)$, which is generated by quadratic paths if and only if
$p=2$. Hence $R_p$ is gentle precisely for $p=1,2$.

Now let $P=\{p,p+1\}$. Lemma~\ref{lem:two-point} gives
$N_p\simeq RQ^{(2)}/((ba)^p)$ and $\dim_R N_p=4p+1$.
If $p=1$, this is a gentle presentation.  Conversely, suppose that $p\ge2$
and that $N_p$ is gentle.  Since $N_p$ is split basic, every gentle
presentation has the same Gabriel quiver $Q^{(2)}$.  Finite dimensionality
forces at least one of the quadratic paths $ba$ and $ab$ to be a relation;
otherwise their alternating powers give nonzero paths of arbitrary length.
Consequently, the only possible nonzero paths are the two trivial paths,
the arrows $a,b$, and at most one of $ba,ab$. Thus the quotient has
dimension at most $5$, contradicting $\dim_R N_p=4p+1\ge9$. Therefore $N_p$
is gentle precisely for $p=1$.

\end{proof}

\subsection{Global criteria and Morita reconstruction}
\label{subsec:global-string}

We now apply Theorem~\ref{thm:local-string} to the primary decomposition.
{For $c\in M_n(R)$ and $p\ge1$, define}
$u_p(c):=\#\{f\in\Irr(c)\mid \deg f=1,\ P_c(f)=\{p\}\}$ and
$v_p(c):=\#\{f\in\Irr(c)\mid
\deg f=1,\ P_c(f)=\{p,p+1\}\}$.
{The algebra
$B_c:=\prod_{f\in\Irr(c)}\Lambda_f(P_c(f))$ is a basic Morita
representative of $S_n(c,R)$ by
\eqref{eq:primary-product}--\eqref{eq:additive-generator-reduction}.}

\begin{corollary}
\label{cor:global-rigidity}
Let $c\in M_n(R)$. Then $S_n(c,R)$ is Morita-string if and only if every
$f\in\Irr(c)$ is linear and every $P_c(f)$ is a singleton or a consecutive
pair. In this case,
\[
 {B_c}\simeq
 \prod_{\substack{p\ge1\\u_p(c)>0}}R_p^{\,u_p(c)}\times
 \prod_{\substack{p\ge1\\v_p(c)>0}}N_p^{\,v_p(c)}.
\]
Both products are finite.
Moreover, $S_n(c,R)$ is Morita-gentle if and only if every $f$ is linear
and $P_c(f)\in\{\{1\},\{2\},\{1,2\}\}$.
Equivalently, its basic algebra has the form
\[
 {B_c}\simeq
 R^{u_1(c)}\times R_2^{u_2(c)}\times N_1^{v_1(c)}.
\]
\end{corollary}

\begin{proof}
{By Theorem~\ref{thm:principal-hh-reconstruction},
$Z(\Lambda_f(P_c(f)))\simeq U_c(f)$ is local. Hence these factors are
exactly the blocks of $B_c$. In a bound-quiver presentation of a basic
algebra, the blocks correspond to the connected components of the quiver.}
By Lemma~\ref{lem:morita-monomial-transfer}, $S_n(c,R)$ is Morita-string
(respectively, Morita-gentle) exactly when $B_c$ is string (respectively,
gentle). {Each block inherits this property from its connected
component.} Conversely, products preserve it by taking disjoint unions
of bound quivers. Theorem~\ref{thm:local-string} therefore gives both
criteria and the displayed decompositions.
\end{proof}

{We now prove the reconstruction statement in
Theorem~\ref{thm:collapse}\textup{(3)}.}

\begin{theorem}
\label{thm:collapse-body}
Let $c\in M_n(R)$ and $d\in M_m(R)$.
\begin{enumerate}[label=\textup{(\arabic*)}]
\item If $S_n(c,R)$ and $S_m(d,R)$ are Morita-string, then
\[
 \mathrm{M}\Longleftrightarrow\mathrm{AD}\Longleftrightarrow
 \mathrm{D}\Longleftrightarrow\mathrm{gTZ}\Longleftrightarrow
 \mathrm{bTZ}.
\]
Within this class, both $\mathrm{bTZ}\Rightarrow\mathrm{TZ}$ and
$\mathrm{TZ}\Rightarrow\mathrm{T}$ are strict.
\item If $S_n(c,R)$ and $S_m(d,R)$ are Morita-gentle, then
$c\simTZ d$ if and only if $S_n(c,R)$ and $S_m(d,R)$ are Morita
equivalent.
\end{enumerate}
\end{theorem}

\begin{proof}
\textup{(1)} The forward implications follow from
Theorem~\ref{thm:strict-comparison} and \cite[Theorem~1.1]{LiXiD}. Suppose
that $c\simbTZ d$, and choose a bijection
$\sigma\colon\Irr(c)\to\Irr(d)$ as in the definition of
bTZ-equivalence. For each $f$, write
$r:=\operatorname{LL}(U_c(f))
=\operatorname{LL}(U_d(\sigma(f)))$ and
{$s:=\kappa_c(f)=\kappa_d(\sigma(f))$}.
By Corollary~\ref{cor:global-rigidity}, both $f$ and $\sigma(f)$ are linear,
$s\in\{1,2\}$, and each exponent set is a singleton or a consecutive pair.
Since $r$ is the largest exponent in each set,
\[
 P_c(f)=P_d(\sigma(f))=
 \begin{cases}
  \{r\},&s=1,\\
  \{r-1,r\},&s=2.
 \end{cases}
\]
Thus $c\simM d$, and the five relations coincide.

To separate bTZ from TZ, take
$c_1:=J_1(0)\oplus J_2(0)\oplus J_3(1)$ and
$c_2:=J_2(0)\oplus J_2(1)\oplus J_3(1)$.
The basic representatives of their centralizer matrix algebras are
$N_1\times R_3$ and $R_2\times N_2$.
Both have center $R_2\times R_3$ and support $\tau$-tilting poset
$\Weak(\Sigma_3)\times\Weak(\Sigma_2)$, whereas their blockwise data
are $\ms{(R_2,2),(R_3,1)}$ and $\ms{(R_2,1),(R_3,2)}$.
Hence $c_1\simTZ c_2$ but $c_1\not\simbTZ c_2$.

To separate TZ from T, choose distinct $\lambda_1,\lambda_2\in R$ and let
$d_1,d_2$ have elementary-divisor sets
$\{x-\lambda_1,(x-\lambda_2)^3\}$ and
$\{(x-\lambda_1)^2,(x-\lambda_2)^2\}$, respectively. Both have T-type $\ms{1,1}$, but their centers are
$R\times R_3$ and $R_2\times R_2$. Thus $d_1\simT d_2$ and
$d_1\not\simTZ d_2$.

\textup{(2)} {Let $S_n(c,R)$ and $S_m(d,R)$ be Morita-gentle},
and suppose $c\simTZ d$. By Corollary~\ref{cor:global-rigidity}, write
their basic algebras as
\[
 {B_e}\simeq
 R^{u_1(e)}\times R_2^{u_2(e)}\times N_1^{v_1(e)}
 \ \ (e\in\{c,d\}).
\]
The support $\tau$-tilting poset, equivalently $T_R(e)$, determines
$v_1(e)$ and $u_1(e)+u_2(e)$. Moreover,
$Z({B_e})\simeq
R^{u_1(e)}\times R_2^{u_2(e)+v_1(e)}$.
Uniqueness of the local center decomposition, together with the counts
determined by the T-type, gives {$u_1(c)=u_1(d)$,
$u_2(c)=u_2(d)$, and $v_1(c)=v_1(d)$}. Thus the basic algebras are
isomorphic, proving Morita equivalence. The converse follows from Morita
invariance of centers and support $\tau$-tilting posets.
\end{proof}

\begin{remark}
The additive-generator reduction
\eqref{eq:additive-generator-reduction} removes the elementary-divisor
multiplicities $m_c(f,a)$. Consequently, even within the Morita-string or
Morita-gentle classes, the preceding reconstruction statements do not, in
general, imply matrix similarity or algebra isomorphism.
\end{remark}

\section*{Acknowledgments}

Jiangsheng Hu was supported by the National Natural Science Foundation of
China (Grant No.~12571035). Yu-Zhe Liu was supported by the National Natural
Science Foundation of China (Grant Nos.~12401042 and 12561008), the Science
and Technology Foundation of the Guizhou S\&T Department (Grant
Nos.~VZD[2026]001, ZD[2025]085, and ZK[2024]YiBan066), and the Scientific
Research Foundation of Guizhou University (Grant No.~[2023]16). Tiwei Zhao
was supported by the National Natural Science Foundation of China (Grant
No.~12471036) and the Hubei Provincial Natural Science Foundation of China
(Grant No.~2026AFA094). The authors thank Professor Changchang Xi for
valuable comments.

\medskip
\noindent
\textbf{Use of AI tools.}
After completing the initial draft of this manuscript, we used ChatGPT (OpenAI)
to assist with English-language editing, literature searches, and LaTeX revisions.
We take full responsibility for the final manuscript, including all mathematical
statements, proofs, and references.

\end{document}